\documentclass{article}

\usepackage{latexsym,amsfonts,amssymb,exscale,enumerate}
\usepackage{amsmath,amsthm,amscd}
\usepackage{hyperref}

\usepackage{pstricks}

\usepackage{bbm}

\numberwithin{equation}{section}

\theoremstyle{definition}
\newtheorem{thm}{Theorem}[section]
\newtheorem{cor}[thm]{Corollary}

\newtheorem{lem}[thm]{Lemma}

\newtheorem{prop}[thm]{Proposition}
\newtheorem{defn}[thm]{Definition}

\newtheorem*{thm*}{Theorem}
\newtheorem*{prop*}{Proposition}

\newcommand{\xto}[1]{{\overset{#1}{\longrightarrow}}}
\newcommand{\mxto}[1]{{\overset{#1}{\longmapsto}}}
\newcommand{\eqr}[1]{{\overset{#1}{=}}}

\def \Q {\mathbbm{Q}}
\def \Z {\mathbbm{Z}}
\def \N {\mathbbm{N}}

\def \X {\Z^2}
\def \I {\{1,2\}}

\def \sl {\mathfrak{sl}}

\def \simeqto {\overset{\simeq}{\longrightarrow }}

\def \u {\mathcal{U}_3}
\def \c {\mathcal{C}}
\def \C {\mathbf{C}}
\def \d {\mathcal{D}}
\def \du {\dot {\mathcal{U}}_3}
\def \up {\mathcal{U}_3^+}
\def \um {\mathcal{U}_3^-}
\def \dup {\dot{\mathcal{U}}_3^+}
\def \dum {\dot{\mathcal{U}}_3^-}
\def \bb {{\bf b}}
\def \kk {{\bf k}}
\def \hh {{\bf h}}
\def \b {\mathcal{B}}

\def \hhlm {{_\mu{\bf h}_\lambda}}
\def \kklm {{_\mu{\bf k}_\lambda}}
\def \Klm {{_\mu K_\lambda}}
\def \Hlm {{_\mu H_\lambda}}
\def \dHlm {{_\mu \dot{H}_\lambda}}
\def \Dlm {{_\mu D_\lambda}}
\def \Mlm {{_\mu M_\lambda}}

\def \e {\mathcal{E}}
\def \f {\mathcal{F}}
\def \eia {\e_i^{(a)}}
\def \fia {\f_i^{(a)}}
\def \Eia {E_i^{(a)}}
\def \Fia {F_i^{(a)}}
\def \rE {{\sf E}}
\def \rF {{\sf F}}
\def \rEia {\rE_i^{(a)}}
\def \rFia {\rF_i^{(a)}}

\def \U {{\bf U}_3}
\def \Up {{\bf U}_3^+}

\def \dU {\dot{\bf U}_3}
\def \dUp {\dot{\bf U}_3^+}

\def \AU {{_{\mathcal A}{\bf U}_3}}
\def \AUp {{_{\mathcal A}{\bf U}_3^+}}

\def \AdU {{_{\mathcal A}}{\dot{\bf U}}_3}
\def \AdUp {{_{\mathcal A}{\dot{\bf U}}_3^+}}
\def \AdUm {{_{\mathcal A}{\dot{\bf U}}_3^-}}

\def \dB {\dot{B}}
\def \dW {\dot{W}}
\def \dS {\dot{S}}

\def \Ob{\operatorname{Ob}}
\def \Tr{\operatorname{Tr}}
\def \H{\operatorname{H}}
\def \Span{\operatorname{Span}}
\def \Kar{\operatorname{Kar}}
\def \Ab{\mathbf{Ab}}
\def \Id {{\rm Id}}

\def \q {q^{-1}}

\def \onel {{\bf 1}_\lambda}
\newcommand{\onell}[1]{{\bf 1}_{#1}}

\def \la {\langle}
\def \ra {\rangle}
\def \scs {\scriptstyle}

\input xy
\usepackage[all,dvips]{xy}
\xyoption{line}
\xyoption{arrow}
\xyoption{color}
\SelectTips{cm}{}

\newcommand{\gt}{\xybox{
(-1.5,2)*{};(1.5,2)*{} **\crv{(-1.5,-0.5) & (1.5,-0.5)};
(0,-2)*{};
}}

\newcommand{\gth}{\xybox{
(-2,-4)*{};(2,-4)*{} **\crv{(-2,-1) & (2,-1)} ?(1)*\dir{>} ?(.7)*{\bullet};
(3,0)*{\scs f_3};
(0,1.5)*{\scs \lambda};
(-6,4)*{};
(6,4)*{};
}}

\newcommand{\gtl}{\xybox{
(-2,4)*{};(2,4)*{} **\crv{(-2,1) & (2,1)} ?(0)*\dir{<} ?(.7)*{\bullet};
(3,1)*{\scs f_1};
(-3,-1)*{};(0,-1)*{} **\crv{(-3,1) & (0,1)} ?(.0)*\dir{<};
(0,-1)*{};(-3,-1)*{} **\crv{(0,-3) & (-3,-3)} ?(.1)*{\ast};
(2,-2)*{\scs f_2};
(-6,4)*{};
(6,4)*{};
}}

\newcommand{\gb}{\xybox{
 (0,-2);(0,2); **\dir{-};
 (0,0)*{\bullet};
}}

\newcommand{\gc}{\xybox{
 (-1,-2)*{};(1,2)*{} **\crv{(-1,0) & (1,0)};
 (1,-2)*{};(-1,2)*{} **\crv{(1,0) & (-1,0)};
 (2,0)*{};(-2,0)*{};
}}

\newcommand{\gch}{\xybox{
 (-1,-2)*{};(1,2)*{} **\crv{(-1,0) & (1,0)} ?(1)*\dir{>};
 (1,-2)*{};(-1,2)*{} **\crv{(1,0) & (-1,0)} ?(0)*\dir{<};
 (2,0)*{};(-2,0)*{};
}}

\newcommand{\gcl}{\xybox{
 (-1,-2)*{};(1,2)*{} **\crv{(-1,0) & (1,0)} ?(0)*\dir{<};
 (1,-2)*{};(-1,2)*{} **\crv{(1,0) & (-1,0)} ?(1)*\dir{>};
 (2,0)*{};(-2,0)*{};
}}

\newcommand{\vc}[5]{
\xy
 (0,4.5)*{};
 (-3,4);(3,4); **\dir{.};
 (5,4)*{#3};
 (0,2)*{#1};
 (-3,0);(3,0); **\dir{.};
 (5,0)*{#4};
 (0,-2)*{#2};
 (-3,-4);(3,-4); **\dir{.};
 (5,-4)*{#5};
 (0,-4.5)*{};
\endxy
}

\newcommand{\vcs}[2]{
\xy
 (0,4.5)*{};
 (-3,4);(3,4); **\dir{.};
 (0,2)*{#1};
 (-3,0);(3,0); **\dir{.};
 (0,-2)*{#2};
 (-3,-4);(3,-4); **\dir{.};
 (0,-4.5)*{};
\endxy
}

\newcommand{\vcsl}[2]{
\xy
 (0,4.5)*{};
 (-7,4);(7,4); **\dir{.};
 (0,2)*{#1};
 (-7,0);(7,0); **\dir{.};
 (0,-2)*{#2};
 (-7,-4);(7,-4); **\dir{.};
 (0,-4.5)*{};
\endxy
}

\newcommand{\tp}[4]{
\left\{\xy
 (0,4.5)*{};
 (-3,4);(3,4); **\dir{--};
 (6,4)*{#3};
 (0,2)*{#1};
 (-3,0);(3,0); **\dir{--};
 (6,0)*{#4};
 (0,-2)*{#2};
 (-3,-4);(3,-4); **\dir{--};
 (6,-4)*{#3};
 (0,-4.5)*{};
\endxy\right\}
}

\newcommand{\tps}[2]{
\left\{\xy
 (0,4.5)*{};
 (-3,4);(3,4); **\dir{--};
 (0,2)*{#1};
 (-3,0);(3,0); **\dir{--};
 (0,-2)*{#2};
 (-3,-4);(3,-4); **\dir{--};
 (0,-4.5)*{};
\endxy\right\}
}

\newcommand{\lbub}{\xybox{
  (-3,0)*{};(3,0)*{} **\crv{(-3,4) & (3,4)} ?(.0)*\dir{>};
  (3,0)*{};(-3,0)*{} **\crv{(3,-4) & (-3,-4)} ?(.1)*{\ast};
  (-5,-5)*{}; (5,5)*{};
}}
\newcommand{\rbub}{\xybox{
  (-3,0)*{};(3,0)*{} **\crv{(-3,4) & (3,4)} ?(.0)*\dir{<};
  (3,0)*{};(-3,0)*{} **\crv{(3,-4) & (-3,-4)} ?(.1)*{\ast};
  (-5,-5)*{}; (5,5)*{};
}}
\newcommand{\lbbub}{\xybox{
  (-3,0)*{};(3,0)*{} **\crv{(-3,4) & (3,4)} ?(.0)*\dir{>};
  (3,0)*{};(-3,0)*{} **\crv{(3,-4) & (-3,-4)} ?(.1)*{\bullet};
  (-5,-5)*{}; (5,5)*{};
}}
\newcommand{\rbbub}{\xybox{
  (-3,0)*{};(3,0)*{} **\crv{(-3,4) & (3,4)} ?(.0)*\dir{<};
  (3,0)*{};(-3,0)*{} **\crv{(3,-4) & (-3,-4)} ?(.1)*{\bullet};
  (-5,-5)*{}; (5,5)*{};
}}

\AtEndDocument{\bigskip{\footnotesize%
  \textsc{Universit\"at Z\"urich, Winterthurerstr. 190 CH-8057 Z\"urich, Switzerland} \par  
  \textit{E-mail address:} \texttt{marko.zivkovic@math.uzh.ch} \par
}}

\begin{document}

\date{\today}

\title{Trace decategorification of categorified quantum $\sl_3$}

\author{Marko \v Zivkovi\' c}

\maketitle

\begin{abstract}
We prove that the trace of categorified quantum $\sl_3$ introduced by M.\ Khovanov and A.\ Lauda can also be identified with quantum $\sl_3$, thus providing an alternative way of decategorification.
\end{abstract}

\section{Introduction}

In this paper we are continuing the program of trace decategorification started in \cite{BHZ}. The first plan of categorification of quantum groups consisted in taking the Grothendieck ring of a category to get a quantum group (see \cite{KL1,KL2,KL3,Lau1}). We hope that the same can be done by taking the trace of the same categories. Hence, we do not construct the category again, we prove that the trace of the categorified quantum group is isomorphic to its Grothendieck ring, i.e.\ the quantum group we have categorified.

The trace of a category is defined in Definition \ref{Def} as
\begin{equation*}
 \Tr(\c) = \bigoplus_{x\in \Ob(\c)}\c(x,x)/\Span\{fg-gf\},
\end{equation*}
where $(f,g)$ run through all pairs of morphisms $f\colon
x\longrightarrow y$, $g\colon y\longrightarrow x$ with $x,y\in
\Ob(\c)$. The trace of the 2-category $\C$ is a 1-category whose homs are the traces of the hom categories of $\C$. In graphical calculus, identifying $fg$ with $gf$ can be seen as mapping a diagram to its image on a cylinder which glues its source and target 1-morphism, such as
\begin{equation}
\label{cylinder}
\xy
{\ar (-9,8)*{}; (-9,-8)*{}};
(-3,-8)*{};(3,8)*{} **\crv{(-3,0) & (3,0)}?(1)*\dir{>};
(3,-8)*{};(-3,8)*{} **\crv{(3,0) & (-3,0)}?(1)*\dir{>}?(.8)*{\bullet};
(-12,0)*{};
(6,0)*{};
\endxy\mapsto \xy
(-5,0)*{};
(0,-7)*{};(0,7)*{} **\crv{(2,-7) & (2,7)};
(0,-7)*{};(0,7)*{} **\crv{(-2,-7) & (-2,7)};
(6,-7)*{};(6,7)*{} **[|(2)]\crv{~**\dir{.} (4,-7) & (4,7)};
(6,-7)*{};(6,7)*{} **[|(2)]\crv{(8,-7) & (8,7)}?(.5)*[|(2)]\dir{<};
(12,-7)*{};(18,7)*{} **[|(2)]\crv{(13,-7) & (14,0) & (20,0) & (19,7)}?(.8)*[|(2)]\dir{>};
(18,-7)*{};(12,7)*{} **[|(2)]\crv{(19,-7) & (20,0) & (14,0) & (13,7)}?(.25)*[|(2)]\dir{>}?(.75)*{\bullet};
(12,-7)*{};(12,7)*{} **[|(2)]\crv{~**\dir{.} (10,-7) & (10,7)};
(18,-7)*{};(18,7)*{} **[|(2)]\crv{~**\dir{.} (16,-7) & (16,7)};
(24,-7)*{};(24,7)*{} **\crv{(26,-7) & (26,7)};
(24,-7)*{};(24,7)*{} **\crv{~**\dir{.} (22,-7) & (22,7)};
(0,-7)*{};(24,-7)*{} **\dir{-};
(0,7)*{};(24,7)*{} **\dir{-};.
\endxy
\end{equation}
Therefore, the trace of a 2-category can be seen as a space of cylinders with strands drawn on them.

There are advantages of trace $\Tr$ compared with Grothendieck ring $K_0$:
\begin{itemize} 
\item
Trace is defined for a larger class of categories (linear instead of additive).
\item
When hom-spaces are graded abelian groups, the trace usually has a richer structure than $K_0$.
\item
Trace of the category coincides with the trace of its Karoubi envelope, which does not hold for the Grothendieck ring in general.
\end{itemize}

Because of the last point, we can slightly simplify the categories which categorify quantum groups--it is not needed to take the Karoubi envelope $\dot{\mathcal{U}}:=\Kar(\mathcal{U})$ at the end, the category $\mathcal{U}$ before it is equally good as a categorified group. However, the Karoubi envelope is still needed to prove the statement, because we are actually calculating the trace of $\dot{\mathcal{U}}$.

In \cite{BHZ} it is shown that the program works for quantum $\sl_2$, i.e.\ that the trace of the categorified quantum $\sl_2$ is indeed the quantum $\sl_2$. Here, in Theorem \ref{Final} we prove the same for $\sl_3$.

\begin{thm*}
There is an isomorphism
\begin{equation*}
\Tr(\u)\cong\AdU
\end{equation*}
of $\Z[q,\q]$-algebras.
\end{thm*}

Here, $\u$ is the categorified quantum $\sl_3$ defined in \cite{KL3} and $\AdU$ is Lusztig's quantum $\sl_3$.

The relative simplicity of the $\sl_3$ case comes from the existence of the explicit Lusztig's canonical base of the positive part $\Up$ of the quantum $sl_3$ (see \cite{Lus}). Sto\v si\' c in \cite{Sto} found the idecomposable objects in $\dup$, the positive part of categorified quantum $\sl_3$, which categorify Lusztig's canonical base, and which form a strongly upper-triangular bases of $\dup$. After that, the result for the positive part of quantum $\sl_3$ (Theorem \ref{poz}) is straightforward.

\begin{thm*}
There is an isomorphism
\begin{equation*}
\Tr(\up)\cong\AdUp
\end{equation*}
of $\Z[q,\q]$-algebras.
\end{thm*}

Here, $\AdUp$ is the positive part of Lusztig's idempotented form of quantum $\sl_3$.
The strategy for the whole quantum $\sl_3$ uses Proposition \ref{mpp} proved in \cite{BHZ}:

\begin{prop*}
Let $\c$ be a linear category. Let $K\subset \H(\c):=\bigoplus_{x\in \Ob(\c)}\c(x,x)$ be a subgroup.
Assume that there is a linear map
 $\pi\colon \H(\c)\rightarrow K$  with the following properties:
\begin{enumerate}
\item $\pi$ is a projection,
\item $[\pi(f)]=[f]\in\Tr(\c)$ for every $f\in \H(\c)$ and
\item for every $g\in\c(x,y)$ and $h\in\c(y,x)$ ($x,y\in\Ob(\c)$), $\pi(gh)=\pi(hg)$ (trace property);
\end{enumerate}
then $\Tr(\c)$ is isomorphic to $K$.
\end{prop*}

The map $\pi$ in the Proposition project every element $f$ of the endomorphism space $H(\c)=\bigoplus_{x\in \Ob(\c)}\c(x,x)$ of a category $\c$ to the well defined representative of $[f]\in\Tr(\c)$.
In $\u$ we construct the map $\pi$ in Subsection \ref{Construct} by cutting a 2-morphism into a vertical composition $\vcs{f}{g}$ and replacing it with $\vcs{g}{f}$ many times, what can be seen as turning cylinder from \eqref{cylinder} around, thus ensuring that we remain in the same equivalence class of the trace. The goal is to separate the positive and the negative part of $\u$ and to use the results known for the positive part $\up$ (the negative part is isomorphic).

Although the construction of the projection $\pi$ in Subsection \ref{Construct} and the inductive proof of the trace property in Subsection \ref{TraceP} look complicated, they are essentially straightforward. We hope that a similar strategy can be used for the quantum groups $\sl_n$ for $n>3$.

\subsection*{Acknowledgements}
I am very grateful to my coauthors of the paper \cite{BHZ}, namely prof.\ Anna Beliakova who introduced me into the field of categorification, prof.\ Kazuo Habiro who suggested this problem and prof.\ Aaron Lauda who made useful suggestions for improving the draft of this paper. I would also like to thank Institute of Mathematics of the the University of Zurich for supporting me while I was doing this research. This work was partially supported by the Swiss National Science Foundation, grants PDAMP2\_137151, 200021\_150012 and the NCCR SwissMAP.

\section{Traces of linear categories and 2-categories}
In this section we recollect the definition and some properties of traces of linear categories and 2-categories as is done in \cite{BHZ}.

\subsection{Linear categories}
A {\em linear} category (also called $\Ab$-category or {\em
pre-additive} category) is a category enriched in the category $\Ab$
of abelian groups ($\Z$-modules). This means it is a category whose hom-sets are equipped with structures of abelian groups and the composition maps are bilinear (compare \cite[p. 276]{MacLane}).

A {\em linear functor} (also called {\em additive functor}) between
two linear categories $\c$ and $\d$ is a functor $F$ from $\c$ to $\d$ such that for $x,y\in \Ob(\c)$, the map $F\colon \c (x,y)\longrightarrow \d (F(x),F(y))$ is an abelian group homomorphism.

\begin{defn}
\label{Def}
The {\em endomorphism space} $\H(\c)$ of the linear category $\c$ is an abelian group defined as
\begin{equation}
 \H(\c) := \bigoplus_{x\in \Ob(\c)}\c(x,x).
\end{equation}

The {\em trace} $\Tr(\c)$ of the linear category $\c$ is the abelian group 
\begin{equation}
\label{Trdef}
 \Tr(\c) := \H(\c)/\Span\{fg-gf\},
\end{equation}
where $(f,g)$ run through all pairs of morphisms $f\colon
x\longrightarrow y$, $g\colon y\longrightarrow x$ with $x,y\in
\Ob(\c)$. The equivalence class $[f]\in\Tr(\c)$ of the morphism $f\in\c(x,x)$ for $x\in\Ob(\c)$ is called {\em the trace} of the morphism $f$.

Let $F\colon\c\to\d$ be a linear functor from linear category $\c$ to linear category $\d$. The {\em trace} $\Tr_F\colon\Tr(\c)\to\Tr(\d)$ is a morphism defined by
\begin{equation}
\Tr_F([f])=[F(f)]
\end{equation}
for $f\colon x\longrightarrow x$, $x\in\Ob(\c)$. It is easy to verify that the definition does not depend on the choice of $f$ representing $[f]$.
\end{defn}

The trace $\Tr$ gives a functor from the category of (small) linear
categories to the category of abelian groups.

Here we summarize some useful facts:
\begin{itemize}
\item For $f\colon x\to x$ and an isomorphism $a\colon y\simeqto x$, we have $[f]=[a^{-1}fa]$ in $\Tr(\c )$.
\item Equivalence of categories $\c \simeq \d $ induces an isomorphism $\Tr(\c )\cong\Tr(\d )$.
\item There is a natural map $h\colon\Ob(\c)\to\Tr(\c)$ defined by $h(x)=[\Id_x]$. The map $h$ is neither injective nor surjective in general.
\end{itemize}

\subsection{Homogeneously graded additive categories}
An {\em additive} category is a linear category equipped with
a zero object and biproducts, also called direct sums (compare \cite[p. 196]{MacLane}).

\begin{prop}[{\cite[Lemma 3.1]{BHZ}}]
If $\c$ is an additive category, then for $f\colon x\rightarrow x$ and $g\colon y\rightarrow y$, we have
  \begin{equation}
  [f\oplus g]=[f]+[g]
  \end{equation}
in $\Tr(\c)$.
\end{prop}

A {\em Homogeneously graded} additive category is an additive category $\c$ such that $\Ob(\c)=\bigoplus_{i\in\Z}\Ob_i(\c)$ and there exists an invertible functor $q\colon\c\to\c$ such that $q(\Ob_i(\c))=\Ob_{i+1}(\c)$. The functor $q$ is called \emph{degree shift}. A morphism $f\in\c(x,y)$ where $x\in\Ob_i(\c)$, $y\in\Ob_j(\c)$ is said to have \emph{degree} $j-i$.
By composing the addition and degree shift operations we can define a functor $p\colon\c\to\c$ for every $p\in\Z^+[q,\q]$.

If $\c$ is a homogeneously graded additive category, note that endomorphisms are linear combinations of 0-degree morphisms and that $\H(\c)=\bigoplus_{i\in\Z}\H(\c_i)$ where $\c_i$ is the full subcategory of $\c$ with objects $\Ob_i(\c)$. Also holds $\Tr(\c)=\bigoplus_{i\in\Z}\Tr(\c_i)$.

The degree shift $q$ induces an isomorphism $\H(\c)\to\H(\c)$ which maps $\H(\c_i)$ to $\H(\c_{i+1})$. It can be seen as the multiplication with the variable $q$, so $\H(\c)$ has a natural structure of $\Z[q,\q]$-module. Therefore, $\Tr(\c)$ has a natural structure of $\Z[q,\q]$-module too, induced by the trace of the degree shift $\Tr_q$. Clearly, $[p(f)]=p[f]$ for $p\in\Z^+[q,\q]$.



\subsection{Karoubi envelopes}
Let $\c$ be a linear category. An idempotent $f\colon x\rightarrow x$ in $\c$ is said to {\em split} if there is an object $y$ and morphisms $g\colon x\rightarrow y$, $h\colon y\rightarrow x$ such that $hg=f$ and $gh=1_y$.

The {\em Karoubi envelope} $\Kar(\c)$ (also called {\em idempotent} or {\em Cauchy completion}) of $\c$ is the  category whose objects are pairs $(x,e)$ of objects $x\in \Ob(\c)$ and an idempotent endomorphism $e\colon x\rightarrow x$, $e^2=e$, in $\c$ and whose morphisms are $(e,f,e') \colon (x,e)\rightarrow(y,e')$ where $f\colon x\rightarrow y$ in $\c$ is such that $f=e'fe$.  
Composition is induced by the composition in $\c$ and the identity morphism is $e\colon (x,e)\rightarrow(x,e)$.
$\Kar(\c)$ is equipped with a linear category structure.

There is a natural embedding functor $i\colon\c\rightarrow\Kar(\c)$ such that $i(x)=(x,1_x)$ for $x\in \Ob(\c)$ and $i(f\colon x \to y)=(\Id_x,f,\Id_y)$. The Karoubi envelope $\Kar(\c)$ has the universal property: if $F\colon \c \rightarrow \d $ is a linear functor to a linear category $\d$ with split idempotents, then $F$ extends to a functor from $\Kar(\c)$ to $\d$ uniquely up to natural epimorphism \cite[Proposition 6.5.9]{Bor}.

It is easy to see that the Karoubi envelope of an additive category
is additive.

\begin{prop}[{\cite[Proposition 3.2.]{BHZ}}]
\label{Kar}
The map $\Tr(i)\colon \Tr(\c)\longrightarrow \Tr(\Kar(\c))$ induced by $i$ is an isomorphism.
\end{prop}

\subsection{Tools for computing the trace}
In this subsection we collect few results which will be needed later 
in the proofs.

\begin{prop}[{\cite[Proposition 3.3.]{BHZ}}]
\label{Dir}
Let $\c$ be a linear category and $B$ a subset of $\Ob(\c)$ such that every object $x$ of $\c$ is isomorphic to direct sum of elements of $B$. Let $\c|_B$ be the full
subcategory of $\c$ with $\Ob(\c|_B)=B$.
Then $\Tr(\c)$ is isomorphic to $\Tr(\c|_B)$.
\end{prop}

\begin{prop}[{\cite[Proposition 3.4.]{BHZ}}]
\label{mpp}
Let $\c$ be a linear category. Let $K\subset \H(\c)$ be a subgroup.
Assume that there is a linear map
 $\pi\colon \H(\c)\rightarrow K$  with the following properties:
\begin{enumerate}
\item $\pi$ is a projection,
\item $[\pi(f)]=[f]\in\Tr(\c)$ for every $f\in \H(\c)$ and
\item for every $g\in\c(x,y)$ and $h\in\c(y,x)$ ($x,y\in\Ob(\c)$), $\pi(gh)=\pi(hg)$ (trace property);
\end{enumerate}
then $\Tr(\c)$ is isomorphic to $K$.
\end{prop}

\begin{defn}
A linear category $\c$ is said to be {\em upper-triangular} if 
there is a partial order $\leq$ on $\Ob(\c)$ such that $\c(x,y)\neq 0$ implies $x\leq y$.

A {\em strongly upper-triangular} linear category is an upper-triangular linear category $\c$ such that for all $x\in \Ob(\c)$, we have $\c (x,x)\cong \Z$.
\end{defn}

\begin{defn}
Let $\c$ be an additive category, and let $B\subset \Ob(\c)$.
Denote by $\c|_B$ the full subcategory of $\c$ with
$\Ob(\c|_B)=B$.

The set $B$ is called a {\em strongly upper-triangular basis} of
$\c$ if the following two conditions hold.
\begin{enumerate}
\item Every object $x\in\Ob(\c)$ is isomorphic to a direct sum of the elements from $B$.
\item $\c|_B$ is strongly upper-triangular.
\end{enumerate}
\end{defn}

\begin{prop}[{\cite[Proposition 4.7.]{BHZ}}]
\label{Sup}
Let $\c$ be an additive category with a strongly upper-triangular basis $B$. Then
  \begin{equation}
    \Tr\c \cong \Z B.
  \end{equation}
\end{prop}

\subsection{2-categories}
Recall that {\em 2-category} $\C$ if given by
\begin{itemize}
\item a set $\Ob(\C)$ of {\em objects} in $\C$,
\item a category $\C(x,y)$ for $x,y\in \Ob(\C)$ (called hom category),
\item an object $I_x\in \C(x,x)$ for $x\in \Ob(\C)$ and
\item a functor $\circ\colon \C(y,z)\times \C(x,y)\rightarrow \C(x,z)$ for
  $x,y,z\in \Ob(\C)$
\end{itemize}
satisfying associativity and identity axioms. For more details see \cite[Chapter 7]{Bor}. 

A {\em linear 2-category} is a 2-category $\C$ whose hom categories are linear categories, and its linear structure is compatible to the structure of 2-category.

Similarly, an {\em additive 2-category} is a linear 2-category $\C$ whose hom categories are additive categories, and its additive structure is compatible to the structure of 2-category.


  
  

A {\em homogeneously graded additive 2-category} is an additive 2-category $\C$ such that the hom categories are homogeneously graded additive categories.

\begin{defn}
%
The {\em trace} $\Tr(\C)$ of the linear 2-category $\C$ is the linear 1-category with same objects as $\C$ and morphisms $\Tr(\C)(x,y)$ being traces of categories $\Tr(\C(x,y))$ for $x,y\in\Ob(\C)$. The composition is defined as:
\begin{equation}
[f][g]=[f\circ g]
\end{equation}
for $f\in\C(x,y)$, $g\in\C(y,z)$, $x,y,z\in\Ob(\C)$. It is easy to verify that the definition does not depend on the choice of $f$ in $[f]$ and $g$ in $[g]$.
\end{defn}

\section{Quantum $\sl_3$}

In this section we define quantum $\sl_3$ and its positive part. The definitions are based on general definitions from \cite{KL3}.

\subsection{Conventional ${\bf U}_q(\sl_3)$}
The algebra ${\bf U}_q(\sl_3)$ is the associative algebra with unit over $\Q(q)$ generated by $E_i$, $F_i$ and $K_i^{\pm 1}$ for $i\in \I$, with defining relations
\begin{equation}
K_iK_i^{-1} = K_i^{-1}K_i=1, \quad K_iK_j = K_jK_i,
\end{equation}
\begin{equation}
K_iE_jK_i^{-1} = q^{i\cdot j}E_j, \quad K_iF_jK_i^{-1} = q^{i\cdot j}F_j,
\end{equation}
\begin{equation}
E_iF_j-F_jE_i = \delta_{i,j}\frac{K_i-K_i^{-1}}{q-\q},
\end{equation}
\begin{equation}
E_i^2E_j-(q+\q)E_iE_jE_i+E_jE_i^2 = 0\quad \text{ if } j\neq i.
\end{equation}
\begin{equation}
F_i^2F_j-(q+\q)F_iF_jF_i+F_jF_i^2 = 0\quad \text{ if } j\neq i.
\end{equation}
where on $\I$ we define the operation $\cdot$ by putting $i\cdot i=2$ and $i\cdot j=-1$ if $j\neq i$.
For simplicity the algebra ${\bf U}_q(\sl_3)$ is written $\U$.

Define the quantum integer $[a]:=\frac{q^a-q^{-a}}{q-q^{-1}}$. The quantum factorial is then $[a]!:=[a][a-1]\dots[1]$, and the quantum binomial coefficient $\left[ \begin{array}{c}a\\b\end{array} \right] := \frac{[a]!}{[b]![a-b]!}$ for $0\leq a \leq b$. For $a\geq 0$, $i\in \I$ we define the divided powers
\begin{equation}
\label{div}
\Eia:=\frac{E_i^a}{[a]!}, \quad
\Fia:=\frac{F_i^a}{[a]!}.
\end{equation}

Let $\AU$ be the $\Z[q,\q]$-subalgebra of $\U$ generated by
\begin{equation*}
\left\{ \Eia,\Fia,K_i^{\pm 1}|a\in\Z_+, i\in \I \right\}.
\end{equation*}

\subsection{Lusztig's quantum $\sl_3$}
The $\Q(q)$-algebra $\dU$ is obtained from $\U$ by replacing the unit with a collection of orthogonal idempotents $1_\lambda$ for $\lambda = (\lambda_1,\lambda_2) \in \X$, i.e.\
\begin{equation*}
1_\lambda 1_\mu = \delta_{\lambda,\mu}1_\lambda,
\end{equation*}
such that
\begin{equation*}
K_i1_\lambda = 1_\lambda K_i = q^{\lambda_i}i_\lambda, \quad, E_i1_\lambda = 1_{\lambda + i'}E_i, \quad F_i1_\lambda = 1_{\lambda - i'}F_i,
\end{equation*}
where
\begin{equation}
\label{39}
1' = (2,-1), \quad 
2' = (-1,2).
\end{equation}
More precisely, $\dU$ is the associative algebra (without unit) over $\Q(q)$ generated by $1_\lambda$, $E_i1_\lambda$ and $F_i1_\lambda$ for $i\in \I$ and $\lambda \in  \X$, with defining relations
\begin{equation}
\label{relI}
1_\lambda 1_\mu = \delta_{\lambda,\mu}1_\lambda,
\end{equation}
\begin{equation}
\label{relII}
E_i1_\lambda = 1_{\lambda + i'}E_i, \quad F_i1_\lambda = 1_{\lambda - i'}F_i,
\end{equation}
\begin{equation}
\label{EFsw}
E_iF_j1_\lambda-F_jE_i1_\lambda = \delta_{i,j}[\lambda_i]1_\lambda,
\end{equation}
\begin{equation}
\label{relEE}
E_i^2E_j1_\lambda-(q+\q)E_iE_jE_i1_\lambda+E_jE_i^21_\lambda = 0 \quad \text{ if } j\neq i.
\end{equation}
\begin{equation}
\label{relFF}
F_i^2F_j1_\lambda-(q+\q)F_iF_jF_i1_\lambda+F_jF_i^21_\lambda = 0 \quad \text{ if } j\neq i.
\end{equation}

Similarly, let $\AdU$ be $\Z[q,\q]$-subalgebra of $\dU$ generated by
\begin{equation*}
\left\{ \Eia 1_\lambda,\Fia 1_\lambda| a\in\Z_+;\, i\in\I;\, \lambda\in \X \right\}.
\end{equation*}

There are direct sum decompositions of algebras
\begin{equation}
\label{DSDL}
\dU = \bigoplus_{\lambda,\mu\in \X}1_\mu \dU 1_\lambda, \quad \quad
\AdU = \bigoplus_{\lambda,\mu\in \X}1_\mu (\AdU) 1_\lambda.
\end{equation}

The algebra $\dU$ does not have a unit since the infinite sum $\sum_{\lambda\in \X}1_\lambda$ is not an element in $\dU$; however, the system of idempotents $\{1_\lambda|\lambda\in \X\}$ in some sense serves as a substitute for a unit. Algebras with system of idempotents have a natural interpretation as linear categories. In this interpretation, $\dU$ is a category with one object $\lambda$ for each $\lambda\in \X$ with homs from $\lambda$ to $\mu$ given by the abelian group $1_\mu\dU 1_\lambda$. The idempotents $1_\lambda$ are the identity morphisms for this category and composition is given by algebra structure of $\dU$. A similar interpretation of algebra $\AdU$ as a linear category also holds.

The following identities hold in $\dU$ and $\AdU$:
\begin{equation}
 E_i^{(a)}E_i^{(b)}1_\lambda=\left[ \begin{array}{c}a+b\\b\end{array}\right]E_i^{(a+b)}1_\lambda,
\end{equation}
\begin{equation}
 F_i^{(a)}F_i^{(b)}1_\lambda=\left[ \begin{array}{c}a+b\\b\end{array}\right]F_i^{(a+b)}1_\lambda,
\end{equation}
\begin{equation}
\label{EFsw1}
 E_i^{(a)}F_i^{(b)}1_\lambda=\sum_{j=0}^{\min(a,b)}\left[ \begin{array}{c}a-b+\lambda_i\\j\end{array}\right]F_i^{(b-j)}E_i^{(a-j)}1_\lambda,\quad\text{for }\lambda_i\geq b-a,
\end{equation}
\begin{equation}
\label{EFsw2}
 F_i^{(b)}E_i^{(a)}1_\lambda=\sum_{j=0}^{\min(a,b)}\left[ \begin{array}{c}b-a-\lambda_i\\j\end{array}\right]E_i^{(a-j)}F_i^{(b-j)}1_\lambda,\quad\text{for }\lambda_i\leq b-a,
\end{equation}
\begin{equation}
\label{EFsw3}
 E_i^{(a)}F_j^{(b)}1_\lambda=F_j^{(b)}E_i^{(a)}1_\lambda \quad\text{if }i\neq j
\end{equation}
for $i,j\in \I$, $\lambda\in \X$.

\subsection{Positive part of quantum $\sl_3$}
In this subsection we define the positive part of quantum $\sl_3$ and recall some results from \cite[Chapter 2]{Sto} needed in the proofs later. Everything holds for the negative part analogously.

The positive part $\Up$ of $\U$ is the subalgebra of $\U$ generated by $\{E_1,E_2\}$.

$\AUp$ is the $\Z[q,\q]$-subalgebra of $\Up$ (or, equivalently, subalgebra of $\AU$) generated by
\begin{equation*}
\left\{ E_1^{(a)},E_2^{(a)}|a\in\Z^+\right\}.
\end{equation*}

We can write Lusztig's canonical basis of $\AUp$ (and of $\Up$) explicitly (see \cite{Lus} or \cite{Lus2}):
\begin{equation*}
B^+:=\left\{E_1^{(a)}E_2^{(b)}E_1^{(c)},E_2^{(a)}E_1^{(b)}E_2^{(c)}|b\leq a+c;\, a,c\geq 0\right\}.
\end{equation*}
One of the remarkable properties of this basis is that its structure constants belong to $\N[q,\q]$, as shown in the next proposition

\begin{prop}[{\cite[Lemma 42.1.2.(d)]{Lus2}}]
For any three nonnegative integers $a,b,c$ with $b\leq a+c$
\begin{equation}
\label{Lrel}
E_1^{(a)}E_2^{(b)}E_1^{(c)} = \sum_{{p+r=b \atop p\leq c} \atop r\leq a}\left[\begin{array}{c}a+c-b\\c-p\end{array}\right]E_2^{(p)}E_1^{(a+c)}E_2^{(r)}.
\end{equation}
In particular, for $b=a+c$
\begin{equation}
\label{EqBase}
E_1^{(a)}E_2^{(a+c)}E_1^{(c)} = E_2^{(c)}E_1^{(a+c)}E_2^{(a)}.
\end{equation}
Finally, both formulas are valid when $E_1$ and $E_2$ interchange places.
\end{prop}

The positive part $\dUp$ of $\dU$ is subalgebra of $\dU$ generated by $E_11_\lambda$, $E_21_\lambda$ for $\lambda\in \X$. Similarly, the positive part $\AdUp$ of $\AdU$ is subalgebra of $\AdU$ generated by $E_1^{(a)}1_\lambda$, $E_2^{(a)}1_\lambda$ for $a\in\Z^+$, $\lambda\in \X$.

There are direct sum decompositions of algebras
\begin{equation}
\dUp = \bigoplus_{\lambda,\mu\in \X}1_\mu \dUp 1_\lambda, \quad \quad
\AdUp = \bigoplus_{\lambda,\mu\in \X}1_\mu (\AdUp) 1_\lambda .
\end{equation}

Let
\begin{equation}
\label{LCBp}
\dB^+:=\left\{b1_\lambda|b\in B^+,\lambda\in \X\right\}.
\end{equation}
Clearly, $\dB^+$ is Lusztig's canonical basis of $\AdUp$ (and of $\dUp$). Relation \eqref{Lrel} holds if we add $1_\lambda$ on the right. For any $\lambda,\mu\in \X$ we define
\begin{equation}
_\mu \dB^+_\lambda:=\dB^+\cap 1_\mu \dUp 1_\lambda.
\end{equation}

Similarly we define $B^-$, $\dB^-$ and $_\mu \dB^-_\lambda$ for the negative part.

\subsection{The basis of $\AdU$}
\label{basis}


Let
\begin{equation}
\dB:=\{b^-b^+|b^-\in \dB^-,b^+\in \dB^+\}-\{0\}\subset\AdU.
\end{equation}

\begin{thm}
$\dB$ is a basis of $\AdU$ as a $\Z[q,\q]$-module.
\end{thm}
\begin{proof}
Using relations \eqref{EFsw1}--\eqref{EFsw3} one can write every element $x\in\AdU$ as a linear combination of elements of the form $x^-x^+$ for $x^-\in\AdUm$, $x^+\in\AdUp$, which are linear combinations of the elements of the form $b^-b^+$ for $b^\pm\in \dB^\pm$. So, $\dB$ generates $\AdU$ as a $\Z[q,\q]$-module.

To finish the proof we need to prove that $\dB$ is linearly independent. It is enough to prove it in the $\Q(q)$-vector space $\dU$. For simplicity, let us deal with the basis $\dB'$ which consists of the multiples of the elements of $\dB$ such that divided powers are replaced with actual ones, i.e.\ we have $F_1^{a}F_2^{b}F_1^{c}E_1^{d}E_2^{e}E_1^{f}1_\lambda$ instead of $F_1^{(a)}F_2^{(b)}F_1^{(c)}E_1^{(d)}E_2^{(e)}E_1^{(f)}1_\lambda$.

Let $W$ be the free $\Q(q)$-algebra over $\{E_1,E_2,F_1,F_2\}$, i.e.\ the $\Q(q)$-algebra with basis consisting of all words in $\{E_1,E_2,F_1,F_2\}$ and concatenation as multiplication. $\dW$ is obtained from $W$ by adjoining idempotents $1_\lambda$ for $\lambda\in \X$, i.e.\ $\dot{W}$ has a basis consisting of all words in $\{E_1,E_2,F_1,F_2\}$ with one letter of the form $1_\lambda$. On $\dW$ we define multiplication as a concatenation up to relations \eqref{relI} and \eqref{relII}, i.e.\ concatenation if the idempotents match, and $0$ otherwise. There is the natural algebra epimorphism $e\colon \dW\to\dU$.

Let $\dS^+\subset\dW$ be generated only with $\{E_1,E_2\}$, and $\dS^-\subset\dW$ with $\{F_1,F_2\}$, together with an idempotent $1_\lambda$. Let $\dS$ be submodule of $\dW$ over the bases $\{b^-b^+|b^\pm\text{ word in }\dS^\pm\}-\{0\}$, i.e.\ generated by words in which all $F$-s are before all $E$-s.

Let the complexity $c$ of a word in $\dW$ be the number of pairs of letters $E$ and $F$ in it, regardless of the index, where $E$ comes before $F$. The complexity of an $x\in\dW$ is the maximum complexity of the word in the representation of $x$ as a linear combination of words. It holds $\dS=\{x\in\dW|c(x)=0\}$.

We define the map $s\colon\dW\to\dW$ on a word $x$ as follows: if $x\in\dS$, $s(x)=x$; if there is $E$ before $F$ in $x$, s replaces the first sub-word $E_iF_j$ with $F_jE_i1_\lambda + \delta_{i,j}[\lambda_i]1_\lambda$ (idempotent is first moved to that place), using relation \eqref{EFsw}. Therefore, $e(s(x))=e(x)$ in $\dU$ for every word $x$. Map $s$ is linearly extended to whole $\dW$. By repeating map $s$ we will finally end up in $\dS$. So, we can define a projection $s^\infty\colon\dW\to\dS$ which is equal to $s^n$ for sufficiently large $n$.

In computing $s^\infty$ we first resolve first sub-word $EF$. But it does not matter which pair is chosen, as shown in the following lemma.

\begin{lem}
\label{33}
For every word $x$ it holds $s^\infty(x)=s^\infty(x')$ if $x'$ is obtained from $x$ by replacing any sub-word $E_iF_j1_\lambda$ with $F_jE_i1_\lambda + \delta_{i,j}[\lambda_i]1_\lambda$.
\end{lem}
\begin{proof}
We prove the statement by induction on the complexity. If the complexity is $1$ there is only one pair $EF$ and therefore $x'=s(x)$, hence the result follows.

Let the complexity be greater and let $x'$ be obtained from $x$ by resolving $EF$ which is not first. Write
\begin{equation*}
x=aEFbEFc,\quad\quad x'=aEFbFEc+k\cdot aEFbc
\end{equation*}
where $a$ is a word without sub-word $EF$, second $EF$ is the one being resolved and $k$ is the constant $\delta_{i,j}[\lambda_i]$. We are a bit imprecise and we omit idempotents. It holds
\begin{equation*}
s(x)=aFEbEFc+k'\cdot abEFc,
\end{equation*}
\begin{equation*}
s(x')=aFEbFEc+k\cdot aFEbc + k'\cdot abFEc+k'k\cdot abc.
\end{equation*}
Note that $s(x')$ is obtained from $s(x)$ by resolving $EF$ in words. Because $s(x)$ has smaller complexity than $x$, by the induction hypothesis $s^\infty(s(x'))=s^\infty(s(x))$ implying that $s^\infty(x')=s^\infty(x)$.
\end{proof}

On $\dS$ there is a multiplication: for $x,y\in\dS$, $x*y:=s^\infty(xy)$. Lemma \ref{33} implies that this multiplication is associative. Clearly, the restriction $e\colon\dS\to\dU$ is a $\Q(q)$-algebra epimorphism.

Let $V^\pm$ be the submodule of $\dS^\pm$ generated by elements of the basis $\dB^\pm$ as words. We choose first word representation of \eqref{EqBase} as a default. Clearly $V^\pm$ is isomorphic to $\dU^\pm$ as a module. We define multiplication on $V^\pm$ taken from $\dU^\pm$. Than there is a natural algebra epimorphisms $q^\pm\colon\dS^\pm\to V^\pm$.

Recall that the words in $\dS$ are of the form $b^-b^+$ for $b^\pm$ word in $\dS^\pm$. Let $V\subset\dS$ be generated with $b^-b^+$ for $b^\pm$ word in $V^\pm$. By construction, $\dB$ (with elements as words) is the basis of $V$. We define epimorphism $q\colon\dS\to V$, $b^-b^+\mapsto q^-(b^-)q^+(b^+)$ and multiplication in $V$ $x\bullet y:=q(x*y)$.

By construction, for every word $x\in\dS$ $e(x)=e(q(x))$ in $\dU$. So, restriction $e\colon V\to\dU$ is a $\Q(q)$-algebra epimorphism. It is easy to check that $V$ satisfies all defining relations \eqref{relI}--\eqref{relFF} of $\dU$. So, by the definition of $\dU$ as a universal algebra with that relations, the epimorphism $e\colon V\to\dU$ has to be an isomorphism. Hence, $\dB$ is the basis of $\dU$, and therefore linearly independent.
\end{proof}

Note that $\dB$ is not the Lusztig's canonical basis. The set $_\mu \dB_\lambda:=\dB\cap 1_\mu\AdU 1_\lambda$ is a basis of $1_\mu\left (\AdU\right) 1_\lambda$.

\section{The $2$-category $\u$}

\subsection{Definition}
\label{DefU}
In this subsection we define the 2-category $\u$, the categorified quantum $\sl_3$ introduced by M.\ Khovanov and A.\ Lauda in \cite{KL3}. We use the standard graphical calculus of string diagrams to describe 2-categories as explained for example in \cite[Section 4]{Lau1}.

The 2-category $\u=\mathcal{U}(\sl_3)$ is the homogeneously graded additive 2-category consisting of
\begin{itemize}
\item objects $\lambda$ for $\lambda \in  \X$,
\item 1-morphisms are formal direct sums of compositions of
\begin{gather*}
\onel \colon \lambda \rightarrow \lambda \quad \text{identity}, \\
\onell{\lambda+i'} \e_i = \onell{\lambda+i'} \e_i\onel = \e_i \onel \colon \lambda \rightarrow \lambda+i' \quad \text{and} \\
\onell{\lambda-i'} \f_i = \onell{\lambda-i'} \f_i\onel = \f_i\onel \colon \lambda \rightarrow \lambda-i'
\end{gather*}
for $\lambda \in  \X$, $i\in \I$, $i'$ defined in \eqref{39}, together with their grading shift $x\la t\ra$ for all 1-morphisms $x$ and $t \in \Z$, and
\item 2-morphisms are free $\Z$-modules generated by (vertical and horizontal) compositions of the following diagrams
\begin{align*}
 \xy
 (0,4);(0,-4); **\dir{-} ?(1)*\dir{>};
 (0,0)*{\bullet};
 (1,-4)*{ \scs i};
 (4,2)*{ \scs \lambda};
 (-6,2)*{\scs \lambda+i'};
 (-6,0)*{};(6,0)*{};
 \endxy &\colon \e_i\onel\la t \ra \to \e_i\onel\la t+2 \ra   
 & \quad &
 \xy
 (0,4);(0,-4); **\dir{-} ?(0)*\dir{<};
 (0,0)*{\bullet};
 (1,4)*{ \scs i};
 (4,2)*{ \scs \lambda};
 (-6,2)*{\scs  \lambda-i'};
 (-6,0)*{};(6,0)*{};
 \endxy \colon \f_i\onel\la t \ra \to \f_i\onel\la t+2 \ra
  \\  & & & \\
 \xy
  (0,0)*{\xybox{
    (-3,-4)*{};(3,4)*{} **\crv{(-3,-1) & (3,1)}?(1)*\dir{>} ;
    (3,-4)*{};(-3,4)*{} **\crv{(3,-1) & (-3,1)}?(1)*\dir{>};
    (-2,-4)*{ \scs j};
    (4,-4)*{ \scs i};
     (5,1)*{\scs  \lambda};
     (-7,0)*{};(7,0)*{};
     }};
  \endxy &\colon \e_i\e_j\onel\la t \ra  \to \e_j\e_i\onel\la t-i\cdot j \ra
  &  &
  \xy
  (0,0)*{\xybox{
    (-3,4)*{};(3,-4)*{} **\crv{(-3,1) & (3,-1)}?(1)*\dir{>};
    (3,4)*{};(-3,-4)*{} **\crv{(3,1) & (-3,-1)}?(1)*\dir{>};
    (-2,4)*{ \scs j};
    (4,4)*{ \scs i};
     (5,-1)*{\scs  \lambda};
     (-7,0)*{};(7,0)*{};
     }};
  \endxy \colon \f_i\f_j\onel\la t \ra  \to \f_j\f_i\onel\la t-i\cdot j \ra
   \\ & & & \\
  \xy (0,0)*{\xybox{
     (-3,4)*{};(3,4)*{} **\crv{(-3,-2) & (3,-2)}?(1)*\dir{>};
     (-2,4)*{ \scs i};
    (5,0)*{\scs  \lambda};
    (-7,0)*{};(7,0)*{};
     }};
    \endxy &\colon \onel\la t \ra  \to \f_i\e_i\onel\la t+1+\lambda_i \ra
      &  &
   \xy (0,0)*{\xybox{
     (3,4)*{};(-3,4)*{} **\crv{(3,-2) & (-3,-2)}?(1)*\dir{>};
     (4,4)*{ \scs i};
    (5,0)*{\scs  \lambda};
    (-7,0)*{};(7,0)*{};
     }};
    \endxy \colon \onel\la t \ra  \to\e_i\f_i\onel\la t+1-\lambda_i\ra
   \\  & & & \\
   \xy (0,0)*{\xybox{
     (3,-4)*{};(-3,-4)*{} **\crv{(3,2) & (-3,2)}?(1)*\dir{>};
     (4,-4)*{ \scs i};
    (5,0)*{\scs  \lambda};
    (-7,0)*{};(7,0)*{};
     }};
    \endxy  &\colon \f_i\e_i\onel\la t \ra \to\onel\la t+1+\lambda_i \ra
   &  &
   \xy (0,0)*{\xybox{
     (-3,-4)*{};(3,-4)*{} **\crv{(-3,2) & (3,2)}?(1)*\dir{>};
     (-2,-4)*{ \scs i};
    (5,0)*{\scs  \lambda};
    (-7,0)*{};(7,0)*{};
     }};
    \endxy \colon\e_i\f_i\onel\la t \ra  \to\onel\la t+1-\lambda_i \ra 
\end{align*}
for every $\lambda\in \X$, $i\in \I$, $t\in\Z$. The degree of the 2-morphisms is the difference of the degree of the target and one of the source and can be read from the shift on the right-hand side.
\end{itemize}

Diagrams are read from right to left and bottom to top. The right most-region in our diagrams is usually coloured by $\lambda=(\lambda_1,\lambda_2)$, and we will usually omit it for simplicity.
The identity 2-morphism of the 1-morphism $\e_i\onel$ is represented by an upward oriented line labeled by $i\in \I$. Likewise, the identity 2-morphism of $\f_i\onel$ is represented by a downward oriented line.

It is convenient to introduce the following degree zero 2-morphisms.
\begin{equation}
\xy
  (0,0)*{\xybox{
    (-3,-5)*{};(3,5)*{} **\crv{(-3,-1) & (3,1)}?(1)*\dir{>};
    (-3,5)*{};(3,-5)*{} **\crv{(-3,1) & (3,-1)}?(1)*\dir{>};
    (-1.5,-5)*{ \scs i};
    (-1.5,5)*{ \scs j};
    (-7,0)*{};(7,0)*{};
  }};
\endxy :=
\xy
  (0,0)*{\xybox{
    (-9,8);(-9,-3); **\dir{-};
    (-9,-3)*{};(-3,-3)*{} **\crv{(-9,-7) & (-3,-7)};
    (-3,-3)*{};(3,3)*{} **\crv{(-3,0) & (3,0)};
    (3,-3)*{};(-3,3)*{} **\crv{(3,0) & (-3,0)};
    (-3,8);(-3,3); **\dir{-}?(1)*\dir{>};
    (9,3)*{};(3,3)*{} **\crv{(9,7) & (3,7)};
    (3,-8);(3,-3); **\dir{-};
    (9,-8);(9,3); **\dir{-}?(1)*\dir{>};
    (4.5,-8)*{ \scs i};
    (-7.5,8)*{ \scs j};
    (-13,0)*{};(13,0)*{};
  }};
\endxy,\quad\quad\quad
\xy
  (0,0)*{\xybox{
    (-3,-5)*{};(3,5)*{} **\crv{(-3,-1) & (3,1)}?(0)*\dir{<};
    (-3,5)*{};(3,-5)*{} **\crv{(-3,1) & (3,-1)}?(0)*\dir{<};
    (4.5,-5)*{ \scs i};
    (4.5,5)*{ \scs j};
    (-7,0)*{};(7,0)*{};
  }};
\endxy :=
\xy
  (0,0)*{\xybox{
    (-9,8);(-9,-3); **\dir{-}?(1)*\dir{>};
    (-9,-3)*{};(-3,-3)*{} **\crv{(-9,-7) & (-3,-7)};
    (-3,-3)*{};(3,3)*{} **\crv{(-3,0) & (3,0)};
    (3,-3)*{};(-3,3)*{} **\crv{(3,0) & (-3,0)};
    (-3,8);(-3,3); **\dir{-};
    (9,3)*{};(3,3)*{} **\crv{(9,7) & (3,7)};
    (3,-8);(3,-3); **\dir{-}?(1)*\dir{>};
    (9,-8);(9,3); **\dir{-};
    (-1.5,8)*{ \scs j};
    (10.5,-8)*{ \scs i};
    (-13,0)*{};(13,0)*{};
  }};
\endxy.
\end{equation}

The 2-morphisms satisfy same relations (see \cite{KL3} for more details). Since $\u$ is homogeneously graded, it is enough to state relations with one chosen degree shift of 1-morphisms and it automatically holds for any other degree shift. We chose that the degree shift of the source (lower) 1-morphism in every diagram is $t=0$. If the shift in the written formula is not stated, we consider it to be $t=0$. The relations are the following.

\begin{enumerate}
\item The 1-morphisms $\e_i\onel$ and $\f_i\onel$ are biadjoint, for $i\in \I$, up to a specified degree shift. Moreover, all 2-morphisms are cyclic with respect to this biadjoint structure.
\end{enumerate}

Relation 1. ensures that diagrams related by isotopy represent the same 2-morphism.

We continue with the relations which involve only one type of strands, either $i=1$ or $i=2$. They are exactly the same as in the $\sl_2$ case (see \cite{Lau1}). We postulate all relations both for $i=1$ and $2$  and omit the label $i$ on the strands for simplicity.
 
\begin{enumerate}
\setcounter{enumi}{1}

\item NilHecke relations:
\begin{equation}
\label{DubbleCrossii}
\xy
(-3,0)*{};(3,6)*{} **\crv{(-3,3) & (3,3)}?(1)*\dir{>};
(3,0)*{};(-3,6)*{} **\crv{(3,3) & (-3,3)}?(1)*\dir{>};
(-3,-6)*{};(3,0)*{} **\crv{(-3,-3) & (3,-3)};
(3,-6)*{};(-3,0)*{} **\crv{(3,-3) & (-3,-3)};
(-6,0)*{};(6,0)*{};
\endxy=0
\end{equation}
\begin{equation}
\label{BullCrossii}
\xy
(-3,-5);(-3,5); **\dir{-} ?(0)*\dir{<};
(3,-5);(3,5); **\dir{-} ?(0)*\dir{<};
(-7,0)*{};(7,0)*{};
\endxy =
\xy
  (0,0)*{\xybox{
    (-3,-5)*{};(3,5)*{} **\crv{(-3,-1) & (3,1)}?(1)*\dir{>}?(.3)*{\bullet};
    (3,-5)*{};(-3,5)*{} **\crv{(3,-1) & (-3,1)}?(1)*\dir{>};
    (-7,0)*{};(6,0)*{};
  }};
\endxy -
\xy
  (0,0)*{\xybox{
    (-3,-5)*{};(3,5)*{} **\crv{(-3,-1) & (3,1)}?(1)*\dir{>}?(.7)*{\bullet};
    (3,-5)*{};(-3,5)*{} **\crv{(3,-1) & (-3,1)}?(1)*\dir{>};
    (-6,0)*{};(7,0)*{};
  }};
\endxy=
\xy
  (0,0)*{\xybox{
    (-3,-5)*{};(3,5)*{} **\crv{(-3,-1) & (3,1)}?(1)*\dir{>};
    (3,-5)*{};(-3,5)*{} **\crv{(3,-1) & (-3,1)}?(1)*\dir{>}?(.7)*{\bullet};
    (-7,0)*{};(6,0)*{};
  }};
\endxy -
\xy
  (0,0)*{\xybox{
    (-3,-5)*{};(3,5)*{} **\crv{(-3,-1) & (3,1)}?(1)*\dir{>};
    (3,-5)*{};(-3,5)*{} **\crv{(3,-1) & (-3,1)}?(1)*\dir{>}?(.3)*{\bullet};
    (-6,0)*{};(7,0)*{};
  }};
\endxy
\end{equation}
\begin{equation}
\xy
 (-6,-7)*{};(6,7)*{} **\crv{(-6,-1) & (6,1)}?(1)*\dir{>};
 (6,-7)*{};(-6,7)*{} **\crv{(6,-1) & (-6,1)}?(1)*\dir{>};
 (0,-7)*{};(0,7)*{} **\crv{(0,-4) & (-4,-3) & (-4,3) & (0,4)}?(1)*\dir{>};
 (-8,0)*{};(8,0)*{};
\endxy =
\xy
 (-6,-7)*{};(6,7)*{} **\crv{(-6,-1) & (6,1)}?(1)*\dir{>};
 (6,-7)*{};(-6,7)*{} **\crv{(6,-1) & (-6,1)}?(1)*\dir{>};
 (0,-7)*{};(0,7)*{} **\crv{(0,-4) & (4,-3) & (4,3) & (0,4)}?(1)*\dir{>};
 (-8,0)*{};(8,0)*{};
\endxy
\end{equation}
Using the adjoint structure, the same relations hold for downward arrows.

\item Dotted bubbles of negative degree are zero, so that for all $m \in \Z^+_0$ one has
\begin{equation}
\xy
 (0,0)*{\lbbub};
 (5,4)*{\lambda};
 (6,0)*{};
 (5,-2)*{\scriptstyle m};
\endxy
  = 0
 \qquad  \text{if $m < \lambda_i-1$,} \quad\quad
\xy 
 (0,0)*{\rbbub};
 (5,4)*{\lambda};
 (6,0)*{};
 (5,-2)*{\scriptstyle m};
\endxy = 0\quad
\text{if $m < -\lambda_i -1$.}
\end{equation}
Dotted bubbles of degree zero are equal to the identity 2-morphism:
\begin{equation}
\label{Zerobub}
\xy
 (0,0)*{\lbbub};
 (5,4)*{\lambda};
 (6,0)*{};
 (7,-2)*{\scriptstyle \lambda_i-1};
\endxy
  =  \Id_{\onel} \quad \text{for $\lambda_i \geq 1$,}
  \qquad \quad
\xy 
 (0,0)*{\rbbub};
 (5,4)*{\lambda};
 (6,0)*{};
 (8,-2)*{\scriptstyle -\lambda_i-1};
\endxy  =  \Id_{\onel} \quad \text{for $\lambda_i \leq -1$.}
\end{equation}
\end{enumerate}

We use the following notation for the dotted bubbles:
\begin{equation}
 \xy
 (0,0)*{\lbub};
 (5,4)*{\lambda};
 (6,0)*{};
 (5,-2)*{\scriptstyle m};
\endxy := \xy
 (5,4)*{\lambda};
(0,0)*{\lbbub};
 (9,-2)*{\scriptstyle m+\lambda_i-1}; 
\endxy,
\quad\quad
 \xy
(0,0)*{\rbub};
 (5,4)*{\lambda};
 (6,0)*{};
 (5,-2)*{\scriptstyle m};
\endxy := \xy
 (5,4)*{\lambda};
(0,0)*{\rbbub};
 (9,-2)*{\scriptstyle m-\lambda_i-1}; 
\endxy,
\end{equation}
for $m+\lambda_i-1\geq 0$, respectively $m-\lambda_i-1\geq 0$, so that
\begin{equation}
 \deg\left( \xy
(0,0)*{\lbub};
 (6,0)*{};
 (5,-2)*{\scriptstyle m};
\endxy\right)= \deg\left(\xy
(0,0)*{\rbub};
 (6,0)*{};
 (5,-2)*{\scriptstyle m};
\endxy\right)=2m.
\end{equation}

If $m+\lambda_i-1<0$, respectively $m-\lambda_i-1<0$, the fake bubbles are defined recursively by the homogeneous terms of the equation
\begin{multline}
\label{gras}
 \left(
\xy
(0,0)*{\rbub};
 (7,0)*{};
 (5,-2)*{\scriptstyle 0};
\endxy+\xy
(0,0)*{\rbub};
 (7,0)*{};
 (5,-2)*{\scriptstyle 1};
\endxy t+\dots+\xy
(0,0)*{\rbub};
 (7,0)*{};
 (5,-2)*{\scriptstyle j};
\endxy t^j+\dots
\right)\cdot \\
\left(
\xy
(0,0)*{\lbub};
 (7,0)*{};
 (5,-2)*{\scriptstyle 0};
\endxy+\xy
(0,0)*{\lbub};
 (7,0)*{};
 (5,-2)*{\scriptstyle 1};
\endxy t+\dots+\xy
(0,0)*{\lbub};
 (7,0)*{};
 (5,-2)*{\scriptstyle j};
\endxy t^j+\dots
\right)=\Id_{\onel} 
\end{multline}
and the additional condition
\begin{equation}
 \xy
(0,0)*{\lbub};
 (6,0)*{};
 (5,-2)*{\scriptstyle 0};
\endxy=\xy
(0,0)*{\rbub};
 (6,0)*{};
 (5,-2)*{\scriptstyle 0};
\endxy=\Id_{\onel}.
\end{equation}
One can check that relation \eqref{gras} holds also for the real bubbles.
So we will not distinguish between real and fake bubbles in what follows.

\begin{enumerate}
\setcounter{enumi}{3}
\item There are additional relations for one type of strands:
\begin{equation}
\xy
(0,-8)*{};(0,8)*{} **\crv{(0,4) & (6,4) & (6,-4) & (0,-4)}?(1)*\dir{>};
(-3,0)*{};(8,0)*{};
\endxy = - \sum_{f_1+f_2=-\lambda_i}\xy
 (0,-8)*{};(0,8)*{} **\dir{-} ?(1)*\dir{>} ?(.5)*{\bullet};
(-2,2)*{\scriptstyle f_1};
 (8,0)*{\lbub};
 (13,-2)*{\scriptstyle f_2};
 (-3,0)*{};(15,0)*{};
\endxy\, ,\quad\quad\xy
(0,-8)*{};(0,8)*{} **\crv{(0,4) & (-6,4) & (-6,-4) & (0,-4)}?(1)*\dir{>};
(3,0)*{};(-8,0)*{};
\endxy = \sum_{f_1+f_2=\lambda_i}\xy
 (0,-8)*{};(0,8)*{} **\dir{-} ?(1)*\dir{>} ?(.5)*{\bullet};
(2,2)*{\scriptstyle f_1};
 (-8,0)*{\rbub};
 (-3,-2)*{\scriptstyle f_2};
 (3,0)*{};(-13,0)*{};
\endxy\, ,
\end{equation}
\begin{equation}
\label{DubbleCrossiir1}
\xy
{\ar (0,-8)*{}; (0,8)*{}};
(-3,0)*{};(3,0)*{};
\endxy\xy
{\ar (0,8)*{}; (0,-8)*{}};
(-3,0)*{};(6,0)*{};
\endxy = -\xy
(-3,0)*{};(3,8)*{} **\crv{(-3,4) & (3,4)};
(3,0)*{};(-3,8)*{} **\crv{(3,4) & (-3,4)}?(1)*\dir{>};
(-3,-8)*{};(3,0)*{} **\crv{(-3,-4) & (3,-4)};
(3,-8)*{};(-3,0)*{} **\crv{(3,-4) & (-3,-4)}?(0)*\dir{<};
(-7,0)*{};(7,0)*{};
\endxy + \sum_{f_1+f_2+f_3=\lambda_i-1}\xy
 (3,9)*{};(-3,9)*{} **\crv{(3,4) & (-3,4)} ?(1)*\dir{>} ?(.2)*{\bullet};
(5,6)*{\scriptstyle f_1};
 (0,0)*{\rbub};
(5,-2)*{\scriptstyle f_2};
 (-3,-9)*{};(3,-9)*{} **\crv{(-3,-4) & (3,-4)} ?(1)*\dir{>}?(.8)*{\bullet};
(5,-6)*{\scriptstyle f_3};
 (-8,0)*{};(8,0)*{};
\endxy\, ,
\end{equation}
\begin{equation}
\label{DubbleCrossiir2}
\xy
{\ar (0,8)*{}; (0,-8)*{}};
(-3,0)*{};(3,0)*{};
\endxy\xy
{\ar (0,-8)*{}; (0,8)*{}};
(-3,0)*{};(6,0)*{};
\endxy = -\xy
(-3,0)*{};(3,8)*{} **\crv{(-3,4) & (3,4)}?(1)*\dir{>};
(3,0)*{};(-3,8)*{} **\crv{(3,4) & (-3,4)};
(-3,-8)*{};(3,0)*{} **\crv{(-3,-4) & (3,-4)}?(0)*\dir{<};
(3,-8)*{};(-3,0)*{} **\crv{(3,-4) & (-3,-4)};
(-7,0)*{};(7,0)*{};
\endxy + \sum_{f_1+f_2+f_3=-\lambda_i-1}\xy
 (3,9)*{};(-3,9)*{} **\crv{(3,4) & (-3,4)} ?(0)*\dir{<} ?(.2)*{\bullet};
(5,6)*{\scriptstyle f_1};
 (0,0)*{\lbub};
(5,-2)*{\scriptstyle f_2};
 (-3,-9)*{};(3,-9)*{} **\crv{(-3,-4) & (3,-4)} ?(0)*\dir{<}?(.8)*{\bullet};
(5,-6)*{\scriptstyle f_3};
 (-8,0)*{};(8,0)*{};
\endxy\, .
\end{equation}
\end{enumerate}
The coefficients $f_i$ run in $\Z^+_0$. 

From now on, instead of labeling strands with $i\in \I$, we usually colour strands: red strands represent strands labeled with $i=1$ and blue strands represent strands labeled with $i=2$.

\begin{enumerate}
\setcounter{enumi}{4}
\item Relations with different type of strands. We draw only one possibility of colouring, the relations with opposite colouring also hold.
\begin{equation}
\label{BullCrossij}
\xy
  (0,0)*{\xybox{
    (-3,-5)*{};(3,5)*{} **[red]\crv{(-3,-1) & (3,1)}?(1)*[red]\dir{>}?(.3)*{\bullet};
    (3,-5)*{};(-3,5)*{} **[blue]\crv{(3,-1) & (-3,1)}?(1)*[blue]\dir{>};
    (-7,0)*{};(7,0)*{};
  }};
\endxy =
\xy
  (0,0)*{\xybox{
    (-3,-5)*{};(3,5)*{} **[red]\crv{(-3,-1) & (3,1)}?(1)*[red]\dir{>}?(.7)*{\bullet};
    (3,-5)*{};(-3,5)*{} **[blue]\crv{(3,-1) & (-3,1)}?(1)*[blue]\dir{>};
    (-7,0)*{};(7,0)*{};
  }};
\endxy \quad \quad
\xy
  (0,0)*{\xybox{
    (-3,-5)*{};(3,5)*{} **[red]\crv{(-3,-1) & (3,1)}?(1)*[red]\dir{>};
    (3,-5)*{};(-3,5)*{} **[blue]\crv{(3,-1) & (-3,1)}?(1)*[blue]\dir{>}?(.3)*{\bullet};
    (-7,0)*{};(7,0)*{};
  }};
\endxy =
\xy
  (0,0)*{\xybox{
    (-3,-5)*{};(3,5)*{} **[red]\crv{(-3,-1) & (3,1)}?(1)*[red]\dir{>};
    (3,-5)*{};(-3,5)*{} **[blue]\crv{(3,-1) & (-3,1)}?(1)*[blue]\dir{>}?(.7)*{\bullet};
    (-7,0)*{};(7,0)*{};
  }};
\endxy
\end{equation}
\begin{equation}
\label{DubbleCrossijr}
\xy
(-3,0)*{};(3,6)*{} **[blue]\crv{(-3,3) & (3,3)};
(3,0)*{};(-3,6)*{} **[red]\crv{(3,3) & (-3,3)}?(1)*[red]\dir{>};
(-3,-6)*{};(3,0)*{} **[red]\crv{(-3,-3) & (3,-3)};
(3,-6)*{};(-3,0)*{} **[blue]\crv{(3,-3) & (-3,-3)}?(0)*[blue]\dir{<};
(-7,0)*{};(7,0)*{};
\endxy =
\xy
(-3,-6);(-3,6); **[red]\dir{-} ?(0)*[red]\dir{<};
(3,-6);(3,6); **[blue]\dir{-} ?(1)*[blue]\dir{>};
(-7,0)*{};(7,0)*{};
\endxy \quad \quad
\xy
(-3,0)*{};(3,6)*{} **[blue]\crv{(-3,3) & (3,3)}?(1)*[blue]\dir{>};
(3,0)*{};(-3,6)*{} **[red]\crv{(3,3) & (-3,3)};
(-3,-6)*{};(3,0)*{} **[red]\crv{(-3,-3) & (3,-3)}?(0)*[red]\dir{<};
(3,-6)*{};(-3,0)*{} **[blue]\crv{(3,-3) & (-3,-3)};
(-7,0)*{};(7,0)*{};
\endxy =
\xy
(-3,-6);(-3,6); **[red]\dir{-} ?(1)*[red]\dir{>};
(3,-6);(3,6); **[blue]\dir{-} ?(0)*[blue]\dir{<};
(-7,0)*{};(7,0)*{};
\endxy
\end{equation}
\begin{equation}
\label{DubbleCrossij}
\xy
(-3,0)*{};(3,6)*{} **[blue]\crv{(-3,3) & (3,3)}?(1)*[blue]\dir{>};
(3,0)*{};(-3,6)*{} **[red]\crv{(3,3) & (-3,3)}?(1)*[red]\dir{>};
(-3,-6)*{};(3,0)*{} **[red]\crv{(-3,-3) & (3,-3)};
(3,-6)*{};(-3,0)*{} **[blue]\crv{(3,-3) & (-3,-3)};
(-7,0)*{};(7,0)*{};
\endxy =
\xy
(-3,-6);(-3,6); **[red]\dir{-} ?(0)*[red]\dir{<} ?(.5)*{\bullet};
(3,-6);(3,6); **[blue]\dir{-} ?(0)*[blue]\dir{<};
(-7,0)*{};(7,0)*{};
\endxy +
\xy
(-3,-6);(-3,6); **[red]\dir{-} ?(0)*[red]\dir{<};
(3,-6);(3,6); **[blue]\dir{-} ?(0)*[blue]\dir{<} ?(.5)*{\bullet};
(-7,0)*{};(7,0)*{};
\endxy
\end{equation}
\begin{equation}
\label{R3iji}
\xy
 (-6,-7)*{};(6,7)*{} **[red]\crv{(-6,-1) & (6,1)}?(1)*[red]\dir{>};
 (6,-7)*{};(-6,7)*{} **[red]\crv{(6,-1) & (-6,1)}?(1)*[red]\dir{>};
 (0,-7)*{};(0,7)*{} **[blue]\crv{(0,-4) & (-4,-3) & (-4,3) & (0,4)}?(1)*[blue]\dir{>};
 (-8,0)*{};(8,0)*{};
\endxy -
\xy
 (-6,-7)*{};(6,7)*{} **[red]\crv{(-6,-1) & (6,1)}?(1)*[red]\dir{>};
 (6,-7)*{};(-6,7)*{} **[red]\crv{(6,-1) & (-6,1)}?(1)*[red]\dir{>};
 (0,-7)*{};(0,7)*{} **[blue]\crv{(0,-4) & (4,-3) & (4,3) & (0,4)}?(1)*[blue]\dir{>};
 (-8,0)*{};(9,0)*{};
\endxy =
\xy
(6,7);(6,-7); **[red]\dir{-}?(1)*[red]\dir{>};
(0,7);(0,-7); **[blue]\dir{-}?(1)*[blue]\dir{>};
(-6,7);(-6,-7); **[red]\dir{-}?(1)*[red]\dir{>};
(-9,0)*{};(9,0)*{};
\endxy
\end{equation}
\begin{equation}
\xy
 (-6,-7)*{};(6,7)*{} **\crv{(-6,-1) & (6,1)}?(1)*\dir{>};
 (6,-7)*{};(-6,7)*{} **\crv{(6,-1) & (-6,1)}?(1)*\dir{>};
 (0,-7)*{};(0,7)*{} **\crv{(0,-4) & (-4,-3) & (-4,3) & (0,4)}?(1)*\dir{>};
 (-8,0)*{};(8,0)*{};
(-5,-7)*{\scs i};
(1,-7)*{\scs j};
(7,-7)*{\scs k};
\endxy =
\xy
 (-6,-7)*{};(6,7)*{} **\crv{(-6,-1) & (6,1)}?(1)*\dir{>};
 (6,-7)*{};(-6,7)*{} **\crv{(6,-1) & (-6,1)}?(1)*\dir{>};
 (0,-7)*{};(0,7)*{} **\crv{(0,-4) & (-4,-3) & (-4,3) & (0,4)}?(1)*\dir{>};
 (-8,0)*{};(8,0)*{};
(-5,-7)*{\scs i};
(1,-7)*{\scs j};
(7,-7)*{\scs k};
\endxy \quad \quad \text{unless $i=k\neq j$.}
\end{equation}
\end{enumerate}

\subsection{Karoubi envelope $\du$}
We define the homogeneously graded additive 2-category $\du$ having the same objects as $\u$ and $\du(\lambda,\mu)=\Kar(\u(\lambda,\mu))$ for $\lambda,\mu\in \X$. Natural embedding functors $\u(\lambda,\mu) \to \du(\lambda,\mu)$ combine to form a natural embedding 2-functor $\u\to\du$. Identifying along this embedding, we can consider $\u\subset\du$.

In $\u$ there is an idempotent $e_{i,a}$ defined as follows:
\begin{equation}
 e_{i,a} := \delta_{i,a}D_{i,a}=
 \xy
(-18,-4)*{}; (9.5,-4)*{} **\dir{-};
(9.5,-4)*{}; (9.5,2)*{} **\dir{-}; 
(9.5,2)*{}; (-18,2)*{} **\dir{-}; 
(-18,2)*{}; (-18,-4)*{} **\dir{-}; 
(-7.5,-4)*{}; (-7.5,-8)*{} **\dir{-};
(-16,-4)*{}; (-16,-8)*{} **\dir{-}; 
(2.5,-4)*{}; (2.5,-8)*{} **\dir{-}; 
(7.5,-4)*{}; (7.5,-8)*{} **\dir{-};
(-6.5,-8)*{\scriptstyle i};
(-15,-8)*{\scriptstyle i};
(3.5,-8)*{\scriptstyle i};
(8.5,-8)*{\scriptstyle i};
{\ar (-16,2)*{}; (-16,8)*{}};
{\ar (-7.5,2)*{}; (-7.5,8)*{}};
{\ar (2.5,2)*{}; (2.5,8)*{}};
{\ar (7.5,2)*{}; (7.5,8)*{}};
(-7.5,5)*{\bullet};
(-16,5)*{\bullet};
(-11.5,6)*{\scriptstyle a-2};
(-20,6)*{\scriptstyle a-1};
(2.5,5)*{\bullet};
(-2.5,-6)*{\dots}; (-2.5,5)*{\dots};
(-4.25,-1)*{D_{i,a}};
(-19.5,0)*{};(11,0)*{};
\endxy,
\end{equation}
where $D_{i,a}$ is the longest braid on $a$ $i$-strands, i.e.\ where every strand crosses every other exactly once. The idempotents
$f_{i,a}$ are obtained from $e_{i,a}$ by a $180^\circ$ rotation.

Therefore, for any $\lambda,\mu \in  \X$, $i\in \I$ we can define additional 1-morphisms in $\du$:
\begin{equation}
 \eia\onel\la t\ra := \left(\e_i^a\onel\left\la t-\tfrac{a(a-1)}{2}\right\ra,e_{i,a}\right),
\end{equation}
\begin{equation}
 \fia\onel\la t\ra := \left(\f_i^a\onel\left\la t+\tfrac{a(a-1)}{2}\right\ra,f_{i,a}\right).
\end{equation}

Since quantum integers $[a]$ and factorials $[a]!$ for $a\in\Z^+$ are in $\Z^+[q,\q]$ we can define
\begin{equation}
[a](x)=x\la a-1\ra\oplus x\la a-3\ra\oplus\cdots\oplus x\la 1-a\ra,
\end{equation}
\begin{equation}
[a]!(x)=[a]([a-1](...[1](x)...))=\bigoplus_{j=0}^a(x\la a-1-2j\ra)^{\left(\begin{array}{c}a\\j\end{array}\right)}
\end{equation}
for 1-morphism $x$ in $\du$.

There is a direct sum decomposition in $\du$ (see \cite{KL1,KL2}):
\begin{equation}
\label{deE}
\e_i^a\onel\cong [a]!\left(\eia\onel\right),
\end{equation}
\begin{equation}
\label{deF}
\f_i^a\onel\cong [a]!\left(\fia\onel\right).
\end{equation}

\subsection{Positive part $\up$ and $\dup$}
In this subsection we define the positive part $\up$ and $\dup$ of $\u$ and $\du$ respectively. Everything holds for negative parts $\um$ and $\dum$ analogously.

The 2-category $\up$ is the subcategory of $\u$ generated by upward arrows, or, more precisely,
\begin{itemize}
\item objects are the same as in $\u$, i.e.\ $\Ob(\up)=\Ob(\u)=\X$,
\item 1-morphisms are formal direct sums of compositions of
\begin{gather*}
\onel \colon \lambda \rightarrow \lambda \quad \text{identity}, \\
\onell{\lambda+i'} \e_i = \onell{\lambda+i'} \e_i\onel = \e_i \onel \colon \lambda \rightarrow \lambda+i'
\end{gather*}
for $\lambda \in  \X$, $i\in \I$ together with their grading shift $x\la t\ra$ for all 1-morphisms $x$ and $t \in \Z$, and
\item 2-morphisms are free $\Z$-modules generated by (vertical and horizontal) compositions of the following diagrams
\begin{equation*}
 \xy
 (0,4);(0,-4); **\dir{-} ?(1)*\dir{>};
 (0,0)*{\bullet};
 (1,-4)*{ \scs i};
 (4,2)*{ \scs \lambda};
 (-6,2)*{\scs \lambda+i'};
 (-6,0)*{};(6,0)*{};
 \endxy \colon \e_i\onel\la t \ra \to \e_i\onel\la t+2 \ra   
  \quad \quad
 \xy
  (0,0)*{\xybox{
    (-3,-4)*{};(3,4)*{} **\crv{(-3,-1) & (3,1)}?(1)*\dir{>} ;
    (3,-4)*{};(-3,4)*{} **\crv{(3,-1) & (-3,1)}?(1)*\dir{>};
    (-2,-4)*{ \scs j};
    (4,-4)*{ \scs i};
     (5,1)*{\scs  \lambda};
     (-7,0)*{};(7,0)*{};
     }};
  \endxy \colon \e_i\e_j\onel\la t \ra  \to \e_j\e_i\onel\la t-i\cdot j \ra
\end{equation*}
\end{itemize}

Similarly to the 2-category $\du$, the additive 2-category $\dup$ has the same objects as $\up$ and $\dup(\lambda,\mu)=\Kar(\up(\lambda,\mu))$ for $\lambda,\mu\in \X$.

By identifying all objects $\lambda$ of $\up$, respectively $\dup$, we get a 1-category called $\up$, respectively $\dup$, by Sto\v si\' c in \cite{Sto}. We use our definition because we want $\up$ to be 2-subcategory of $\u$. Some results of Sto\v si\'c in \cite[Chapter 4]{Sto} can be straightforwardly translated to our framework as follows.

\begin{thm}[{\cite[Theorem 2]{Sto}}]
The set of indecomposable 1-morphisms of $\dup(\lambda,\mu)$, $\lambda,\mu\in \X$ is the following set:
\begin{multline*}
\left\{\e_1^{(a)}\e_2^{(b)}\e_1^{(c)}\onel\la t \ra|b\geq a+c;\, a,b,c\geq 0;\, t\in\Z;\, \mu=\lambda + (a+c)1'+b2'\right\}\\
\cup
\left\{\e_2^{(a)}\e_1^{(b)}\e_2^{(c)}\onel\la t \ra|b\geq a+c;\, a,b,c\geq 0;\, t\in\Z;\, \mu=\lambda + (a+c)2'+b1'\right\}
\end{multline*}
No two elements from above are isomorphic, except that
\begin{equation}
\e_1^{(a)}\e_2^{(a+c)}\e_1^{(c)}\onel\la t \ra \cong \e_2^{(c)}\e_1^{(a+c)}\e_2^{(a)}\onel\la t \ra,\quad a,c\geq 0;\, t\in\Z.
\end{equation}
\end{thm}

To have the uniqueness we consider the set of indecomposables without one of isomorphic 1-morphisms:
\begin{multline}
_\mu\b^{*+}_\lambda:=
\left\{\e_1^{(a)}\e_2^{(b)}\e_1^{(c)}\onel\la t \ra|b\geq a+c;\, a,b,c\geq 0;\, t\in\Z;\, \mu=\lambda + (a+c)1'+b2'\right\}\\
\cup
\left\{\e_2^{(a)}\e_1^{(b)}\e_2^{(c)}\onel\la t \ra|b> a+c;\, a,b,c\geq 0;\, t\in\Z;\, \mu=\lambda + (a+c)2'+b1'\right\}
\end{multline}
Let $_\mu\b^+_\lambda$ be the set of elements of $_\mu\b^{*+}_\lambda$ without shift, i.e.\ with shift $t=0$.
Let
\begin{equation}
\b^{*+}:=\bigcup_{\lambda,\mu\in \X}{_\mu\b^{*+}_\lambda},\quad\quad
\b^+:=\bigcup_{\lambda,\mu\in \X}{_\mu\b^+_\lambda}.
\end{equation}
Note that there is a natural bijection $\iota^+\colon \dB^+\to \b^+$ between Lusztig's canonical basis $B^+$ of $\AdUp$ (defined in \eqref{LCBp}) and $\b^+$:
\begin{equation}
\label{iota}
E_1^{(a)}E_2^{(b)}E_1^{(c)}1_\lambda \mxto{\iota^+} \e_1^{(a)}\e_2^{(b)}\e_1^{(c)}\onel, \quad\quad
E_2^{(a)}E_1^{(b)}E_2^{(c)}1_\lambda \mxto{\iota^+} \e_2^{(a)}\e_1^{(b)}\e_2^{(c)}\onel.
\end{equation}

\begin{prop}[{\cite[Proof of Theorem 2]{Sto}}]
\label{sup}
Let $t,t'\in\Z$ and $x,y\in{_\mu\b^{*+}_\lambda}$. If $t-t'<0$  or $t-t'=0$ and $x\neq y$, then
\begin{equation*}
 \dup(x\la t\ra,y\la t'\ra)=\emptyset.
\end{equation*}
The only elements in $\dup(x\la t\ra,y\la t\ra)$ are multiples of the identity.
\end{prop}

\begin{thm}[{\cite[Theorem 4]{Sto}}]
\label{dec}
An arbitrary 1-morphism of $\dup$ can be decomposed as an direct sum of the elements from $\b^{*+}$.
\end{thm}

Now we can conclude the following.

\begin{cor}
\label{TrUpH}
\begin{enumerate}
\item For $\lambda,\mu\in \X$, the set $_\mu\b^{*+}_\lambda$ is a strongly upper-triangular basis of $\dup(\lambda,\mu)$.
\item For $\lambda,\mu\in \X$
\begin{equation}
\Tr(\up(\lambda,\mu))\cong\Tr(\dup(\lambda,\mu))\cong\Z{_\mu\b^{*+}_\lambda}\cong\Z[q,\q]{_\mu\b^+_\lambda}.
\end{equation}
\end{enumerate}
\end{cor}
\begin{proof}
\begin{enumerate}
\item Straightforward from Proposition \ref{sup} and Theorem \ref{dec}.
\item The first isomorphism follows from proposition \ref{Kar}, and the second from (1) and proposition \ref{Sup}. The last one is clear from $q^t [x\la 0\ra]=[x\la t\ra]\in\Tr(\dup(\lambda,\mu))$ for $x\la 0\ra\in{_\mu\b^0_\lambda}$ and $t\in\Z$.
\end{enumerate}
\end{proof}

Let
\begin{equation}
\label{Ei}
\rE_i\onel:=\left[\Id_{\e_i\onel}\right],\quad\quad
\rEia\onel:=\left[\Id_{\eia\onel}\right],
\end{equation}
in $\Tr(\dup)$ for $i\in \I$, $a\in \X$. The isomorphisms from the corollary (2) map
\begin{equation*}
\left[e_{i,a}\colon \e_i^a\onel\left\la t-\tfrac{a(a-1)}{2}\right\ra\to\e_i^a\onel\left\la t-\tfrac{a(a-1)}{2}\right\ra
\right]\mapsto q^t\rEia\onel\mapsto \eia\onel\la t \ra \mapsto q^t\cdot\eia\onel.
\end{equation*}




The next result, the trace of upper part of categorified quantum $\sl_3$, is not necessary because we do not need it later and it follows from the final result. But now it is easy to get it, so we prove it separately.

\begin{thm}
\label{poz}
There is an isomorphism
\begin{equation}
\Tr(\up)\cong\AdUp
\end{equation}
of $\Z[q,\q]$-algebras.
\end{thm}
\begin{proof}
Proposition \ref{Kar} implies that $\Tr(\up)\cong\Tr(\dup)$, so it is enough to prove that $\Tr(\dup)\cong\AdUp$.

Corollary \ref{TrUpH} (2) implies that $\Tr(\dup)$ is generated by $\b^+$ as a $\Z[q,\q]$-module. Since there is a bijection between $\b^+$ and Lusztig's canonical basis $\dB^+$ of $\AdUp$ we have
\begin{equation*}
\Tr(\dup)\cong\AdUp
\end{equation*}
as a $\Z[q,\q]$-module. Let the isomorphism be $[\iota^+]\colon\AdUp\to\Tr(\dup)$. It maps $\Eia 1_\lambda\mapsto\rEia\onel$. To prove the theorem we need to check that elements multiply (or compose if we look to it as a 1-category) in the same way.

$\AdUp\otimes\Q(q)\cong\dUp$ is generated by $E_i1_\lambda$ for $i\in \I$, $\lambda\in \X$ as $\Q(q)$-algebra. Relation \eqref{deE} implies
\begin{equation}
\label{divT}
\rE_i^a\onel = [a]!\rEia\onel
\end{equation}
in $\Tr(\dup)$. Therefore, $\Tr(\dup)\otimes\Q(q)$ is generated by $\rE_i\onel=[\iota^+](E_i1_\lambda)$ for $i\in \I$, $\lambda\in \X$ as $\Q(q)$-algebra. Recall that defining relation of $\dUp$ is \eqref{relEE}:
\begin{equation*}
E_i^2E_j1_\lambda-(q+\q)E_iE_jE_i1_\lambda+E_jE_i^21_\lambda = 0 \quad \text{ if } j\neq i.
\end{equation*}
Lemma \ref{rel1} ensures that same relation holds in $\Tr(\dup)\otimes\Q(q)$ after applying $[\iota^+]$ . Divided powers, from which we construct the bases, are defined from above generators in the same way (compare \eqref{divT} and \eqref{div}) so it is really
\begin{equation*}
\Tr(\dup)\otimes\Q(q)\cong\AdUp\otimes\Q(q).
\end{equation*}
Since the multiplications are the same, they are also the same in the $\Z[q,\q]$-subalgebras $\Tr(\dup)$ and $\AdUp$, hence the proof.
\end{proof}

\section{The result}

In this section we prove that the trace of $\u$ is isomorphic to $\AdU$ as an algebra (or 1-category). First we need some notations.

From the definition, the set of all 1-morphisms of $\u$ is spanned by the set of non-zero compositions of $\e_i\onel$ and $\f_i\onel$, $i\in \I$, together with degree shift $t\in\Z$. Let us call this set of generating 1-morphisms $\hh$. We define the length $l\colon\hh\to\Z$ to be the total number of $\e_i$-s and $\f_i$-s in a generator. For fixed objects $\lambda,\mu\in \X$ let $\hhlm:=\hh\cap\Ob(\u(\lambda,\mu))$.

For a generating 1-morphisms $x,y\in\hhlm$, the space of 2-morphisms $f\colon x\to y$ is free $\Z$-module spanned by the base $\bb(x,y)$ consisted of 2-morphisms with diagrams whose strands have no self intersections, no two strands intersect more than once, all dots are confined to the beginning interval on each strand, and all closed diagrams are dotted bubbles with the same orientation for each $i\in \I$ at the far right of the diagram (see \cite[Subsection 3.2.]{KL3}) such as the diagram on the picture.
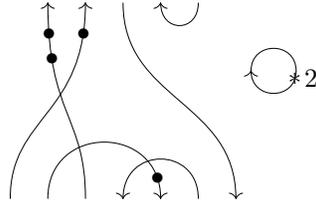
\begin{equation*}
\xy
(0,-13)*{};(10,13)*{} **\crv{(0,-2) & (10,-2)}?(1)*\dir{>} ?(.9)*{\bullet};
(5,-13)*{};(20,-13)*{} **\crv{(5,-3) & (20,-3)}?(1)*\dir{>} ?(.9)*{\bullet};
(10,-13)*{};(5,13)*{} **\crv{(10,-2) & (5,-2)}?(1)*\dir{>} ?(.9)*{\bullet} ?(.8)*{\bullet};
(25,-13)*{};(15,-13)*{} **\crv{(25,-6) & (15,-6)}?(1)*\dir{>};
(15,13)*{};(30,-13)*{} **\crv{(15,0) & (30,0)}?(1)*\dir{>};
(25,13)*{};(20,13)*{} **\crv{(25,9) & (20,9)}?(1)*\dir{>};
(35,4)*{\lbub};
(40,3)*{2}
\endxy
\end{equation*}
We can make sure that every strand has at most one turn, that all turns pointing down are below turns pointing up, and that all crossings of two strands without turns are in between.


We first prove that $\Tr(\u(\lambda,\mu))\cong\AdU(\lambda,\mu)$ as a module. Because of Proposition \ref{Dir}, it is enough to prove that $\Tr(\u(\lambda,\mu)|_{\hhlm})\cong\AdU(\lambda,\mu)$. We do it using Proposition \ref{mpp}. In the first subsection of this section we construct the projection $\pi$ needed in the Proposition. By the construction it will be clear that $\pi$ satisfies first two conditions from the Proposition. In the second subsection we prove that $\pi$ satisfies the condition 3.\ (trace property) too.

In third and the last subsection we check the relations and thus prove the claim, similarly to the proof of the Theorem \ref{poz}.

A 2-morphism $a\colon y\to z$ in 2-category $\C$ is graphically represented by
\begin{equation*}
\xy
 (-6,5);(6,5); **\dir{.};
 (8,5)*{y};
 (0,0)*{a};
 (-6,-5);(6,-5); **\dir{.};
 (8,-5)*{x};
\endxy
\end{equation*}
and a vertical composition, i.e.\ composition in hom category $\C(\lambda,\mu)$, of 2-morphisms $b\colon x\to y$ and $a\colon y\to z$ is represented by
\begin{equation*}
\xy
 (0,9)*{};
 (-6,8);(6,8); **\dir{.};
 (8,8)*{z};
 (0,4)*{a};
 (-6,0);(6,0); **\dir{.};
 (8,0)*{y};
 (0,-4)*{b};
 (-6,-8);(6,-8); **\dir{.};
 (8,-8)*{x};
 (0,-9)*{};
\endxy
\end{equation*}
If it is not essential, we omit 1-morphisms $x$, $y$ and $z$ and sometimes dotted lines.

\subsection{Construction of the projection $\pi$}
\label{Construct}

Let $K^+$ be the abelian group formally spanned by $\{\Id_x|x\in\b^{*+}\}$. It is a $\Z$-module over $\b^{*+}$ and can be seen as $\Z[q,\q]$-module over $\b^+$. There is an isomorphism $\iota^+\colon \AdUp\to K^+$ as a $\Z[q,\q]$-modules which expands bijection $\iota^+$ from \eqref{iota}. Similarly, let $K^-$ be abelian group formally spanned by $\{\Id_x|x\in\b^{*-}\}$. It is a $\Z$-module over $\b^{*-}$ and can be seen as $\Z[q,\q]$-module over $\b^-$. There is an isomorphism $\iota^-\colon \AdUm\to K^-$ as a $\Z[q,\q]$-modules.

Let $K:=\{f^-\circ f^+|f^-\in K^-,f^+\in K^+\}$. Clearly, it is a $\Z[q,\q]$-module over $\b:=\{b^-\circ b^+|b^-\in\b^-,\,b^+\in\b^+\}$. There is a natural bijection $\iota\colon \dB\to\b$ between the basis $\dB$ of $\AdU$ (defined in subsection \ref{basis}) and $\b$:
\begin{equation}
 b^-b^+\mxto{\iota}\iota^-(b^-)\circ\iota^+(b^+),
\end{equation}
which can be extended to isomorphism $\iota\colon \AdU\to K$ as $\Z[q,\q]$-modules.

Let us now fix objects $\lambda,\mu\in\X$. Let $\Klm:=K\cap\H(\u(\lambda,\mu))\cong 1_\mu(\AdU)1_\lambda$. Let
\begin{equation}
\dHlm:=\H(\du(\lambda,\mu)|_{\hhlm}).
\end{equation}
Our goal is to construct the projection $\pi\colon\dHlm\to \Klm$ which satisfies all conditions from Proposition \ref{mpp}. We do it in steps:
\begin{equation}
\label{Maps}
\dHlm \to \Hlm \to \Mlm \to \Dlm \to \Klm
\end{equation}
where $\Hlm:=\H(\u(\lambda,\mu)|_{\hhlm})$, $\Mlm$ is the submodule of $\Hlm$ without turns, i.e.\ generated by those 2-morphism which are generated by
\begin{equation*}
\xy
 (0,4);(0,-4); **\dir{-} ?(1)*\dir{>};
 (0,0)*{\bullet};
 (1,-4)*{ \scs i};
 (4,2)*{ \scs \lambda};
 (-6,0)*{};(6,0)*{};
\endxy \quad 
\xy
 (0,4);(0,-4); **\dir{-} ?(0)*\dir{<};
 (0,0)*{\bullet};
 (1,4)*{ \scs i};
 (4,2)*{ \scs \lambda};
 (-6,0)*{};(6,0)*{};
\endxy \quad
\xy
  (0,0)*{\xybox{
    (-3,-4)*{};(3,4)*{} **\crv{(-3,-1) & (3,1)}?(1)*\dir{>} ;
    (3,-4)*{};(-3,4)*{} **\crv{(3,-1) & (-3,1)}?(1)*\dir{>};
    (-2,-4)*{ \scs j};
    (4,-4)*{ \scs i};
     (7,1)*{\scs  \lambda};
     (-7,0)*{};(7,0)*{};
     }};
\endxy \quad
\xy
  (0,0)*{\xybox{
    (-3,4)*{};(3,-4)*{} **\crv{(-3,1) & (3,-1)}?(1)*\dir{>};
    (3,4)*{};(-3,-4)*{} **\crv{(3,1) & (-3,-1)}?(1)*\dir{>};
    (-2,4)*{ \scs j};
    (4,4)*{ \scs i};
     (7,1)*{\scs  \lambda};
     (-7,0)*{};(7,0)*{};
     }};
\endxy
\end{equation*}
for $i,j\in \I$, $\lambda\in \X$, and $\Dlm$ is the submodule of $\Mlm$ generated by 2-morphisms all of whose downward arrows are left of all upward arrows. More precisely, let $\kk^+\subset\hh$ be the set of 1-morphisms generated by upward arrows, i.e.\ by $\e_i\onel$, $\kk^-\subset\hh$ be the set of 1-morphisms generated by downward arrows and $\kk:=\{x^-x^+\in\hh|x^-\in \kk^-,x^+\in \kk^+\}$. Then $\Dlm\subset \Mlm$ is generated by 2-morphisms $f\colon x\to x$ for $x\in\kklm:=\kk\cap\Ob(\u(\lambda,\mu))$.

For easier understanding, we rewrite \eqref{Maps} using some example elements of the spaces:
\begin{equation*}
\left\{\left(e,
\xy
(0,-8)*{};(6,8)*{} **\crv{(0,-2) & (6,-2)}?(1)*\dir{>} ?(.93)*{\bullet};
(3,-8)*{};(12,-8)*{} **\crv{(3,-2) & (12,-2)}?(1)*\dir{>} ?(.9)*{\bullet};
(6,-8)*{};(3,8)*{} **\crv{(6,-2) & (3,-2)}?(1)*\dir{>};
(15,-8)*{};(9,-8)*{} **\crv{(15,-4) & (9,-4)}?(1)*\dir{>};
(9,8)*{};(18,-8)*{} **\crv{(9,0) & (18,0)}?(1)*\dir{>};
(12,8)*{};(0,8)*{} **\crv{(12,3) & (0,3)}?(1)*\dir{>} ?(.85)*{\bullet};
(18,8)*{};(15,8)*{} **\crv{(18,6) & (15,6)}?(1)*\dir{>};
(-1,-9)*{};(19,9)*{};
\endxy
,e\right)\right\}
\to
\left\{
\xy
(0,-8)*{};(6,8)*{} **\crv{(0,-2) & (6,-2)}?(1)*\dir{>} ?(.93)*{\bullet};
(3,-8)*{};(12,-8)*{} **\crv{(3,-2) & (12,-2)}?(1)*\dir{>} ?(.9)*{\bullet};
(6,-8)*{};(3,8)*{} **\crv{(6,-2) & (3,-2)}?(1)*\dir{>};
(15,-8)*{};(9,-8)*{} **\crv{(15,-4) & (9,-4)}?(1)*\dir{>};
(9,8)*{};(18,-8)*{} **\crv{(9,0) & (18,0)}?(1)*\dir{>};
(12,8)*{};(0,8)*{} **\crv{(12,3) & (0,3)}?(1)*\dir{>} ?(.85)*{\bullet};
(18,8)*{};(15,8)*{} **\crv{(18,6) & (15,6)}?(1)*\dir{>};
(-1,-9)*{};(19,9)*{};
\endxy
\right\}
\to
\left\{
\xy
(0,-8)*{};(6,8)*{} **\crv{(0,0) & (6,0)}?(1)*\dir{>} ?(.9)*{\bullet};
(3,8)*{};(3,-8)*{} **\crv{(3,4) & (5,4) & (5,-4) & (3,-4)}?(1)*\dir{>} ?(.9)*{\bullet};
(6,-8)*{};(0,8)*{} **\crv{(6,0) & (0,0)}?(1)*\dir{>};
(9,8);(9,-8); **\dir{-} ?(0)*\dir{<};
(-1,-9)*{};(10,9)*{};
\endxy
\right\}
\to
\left\{
\xy
(0,8);(0,-8); **\dir{-} ?(0)*\dir{<};
(3,8);(3,-8); **\dir{-} ?(0)*\dir{<};
(6,-8)*{};(9,8)*{} **\crv{(6,0) & (9,0)}?(1)*\dir{>} ?(.85)*{\bullet};
(9,-8)*{};(6,8)*{} **\crv{(9,0) & (6,0)}?(1)*\dir{>};
(-1,-9)*{};(10,9)*{};
\endxy
\right\}
\to
\left\{\f_1^{(2)}\e_2^{(2)}\right\}
\end{equation*}

First map is the natural projection $\alpha\colon\dHlm\to \Hlm$, $\alpha\colon (e,f,e)\mapsto (\Id,f,\Id)=f$. It is clear that $[\alpha(f)]=[f]$ in $\Tr(\du)$.

We define map $\beta\colon \Hlm\to \Hlm$ on base elements $b\in\bb(x,x)$ for $x\in\hhlm$: if $b$ has bubbles $\beta(b)=0$, otherwise we cut the 2-morphism just above turns pointing down, and switch two parts of it, for example:
\begin{equation}
\label{cut}
\xy
(-1,13);(31,13); **\dir{.};
(33,13)*{x};
(-1,-3);(31,-3); **\dir{.};
(35,-3)*{y\la t \ra};
(-1,-13);(31,-13); **\dir{.};
(33,-13)*{x};
(0,-13)*{};(10,13)*{} **\crv{(0,-2) & (10,-2)}?(1)*\dir{>} ?(.9)*{\bullet};
(5,-13)*{};(20,-13)*{} **\crv{(5,-3) & (20,-3)}?(1)*\dir{>} ?(.9)*{\bullet};
(10,-13)*{};(5,13)*{} **\crv{(10,-2) & (5,-2)}?(1)*\dir{>};
(25,-13)*{};(15,-13)*{} **\crv{(25,-6) & (15,-6)}?(1)*\dir{>};
(15,13)*{};(30,-13)*{} **\crv{(15,0) & (30,0)}?(1)*\dir{>};
(20,13)*{};(0,13)*{} **\crv{(20,5) & (0,5)}?(1)*\dir{>} ?(.9)*{\bullet} ?(.8)*{\bullet};
(30,13)*{};(25,13)*{} **\crv{(30,9) & (25,9)}?(1)*\dir{>};
\endxy\quad\quad\mxto{\beta}\quad q^t\quad
\xy
(-1,13);(31,13); **\dir{.};
(33,13)*{y};
(-1,3);(31,3); **\dir{.};
(36,3)*{x\la -t \ra};
(-1,-13);(31,-13); **\dir{.};
(33,-13)*{y};
(4,-13)*{};(10,3)*{} **\crv{(10,-6)} ?(.8)*{\bullet};
(0,3)*{};(4,13)*{} **\crv{(0,8)}?(1)*\dir{>};
(8,-13)*{};(5,3)*{} **\crv{(5,-6)};
(10,3)*{};(8,13)*{} **\crv{(10,8)}?(1)*\dir{>};
(26,13)*{};(30,3)*{} **\crv{(30,9)};
(15,3)*{};(26,-13)*{} **\crv{(15,-6) & (24,-11)}?(1)*\dir{>};
(5,3)*{};(20,3)*{} **\crv{(5,13) & (20,13)} ?(.9)*{\bullet};
(25,3)*{};(15,3)*{} **\crv{(25,10) & (15,10)};
(20,3)*{};(0,3)*{} **\crv{(20,-5) & (0,-5)} ?(.9)*{\bullet} ?(.8)*{\bullet};
(30,3)*{};(25,3)*{} **\crv{(30,-1) & (25,-1)};
\endxy
\end{equation}
Note the degree shift gained from moving bottom 1-morphism back to degree $0$. It is clear that $[\beta(f)]=[f]$ in $\Tr(\du)$, even if there are bubbles, because in that case $[f]=0$.

By repeating the map $\beta$ we will finally end up in $\Mlm$, and $\beta|_\Mlm=\Id_\Mlm$. So we can define a projection $\beta^\infty\colon \Hlm\to \Mlm$ which is equal to $\beta^n$ for sufficiently large $n$. It is our second map of \eqref{Maps}.

We define $g\colon\hhlm\to\hhlm$: if $x\in\kk$, $g(x)=x$; if there is $\f$ ($\f_1$ of $\f_2$) after $\e$ in $x$, $g$ switches the first such $\f$ with $\e$ before it. Similarly, we define $h\colon\hhlm\to\hhlm$: if $x\in\kk$, $h(x)=x$; if there is $\f$ after $\e$ in $x$, $h$ deletes the first such $\f$ and $\e$ before it if they are labelled by the same index $i$, and is $0$ if they are labeled by different indices.

We construct a map $\gamma\colon \Hlm\to \Hlm$ on base elements $f\colon x\to x$:
\begin{itemize}
\item if $x\in\kk$, $\gamma(f)=f$,
\item if $g$ switches $\f_j$ and $\e_i$ for $i\neq j$
\begin{equation}
\label{gamma1}
\xy
 (-6,5);(6,5); **\dir{.};
 (8,5)*{x};
 (0,0)*{f};
 (-6,-5);(6,-5); **\dir{.};
 (8,-5)*{x};
\endxy\quad\mxto{\gamma}\quad
\xy
 (-6,8);(6,8); **\dir{.};
 (10,8)*{g(x)};
 (0,6)*{\gch};
 (-6,4);(6,4); **\dir{.};
 (8,4)*{x};
 (0,0)*{f};
 (-6,-4);(6,-4); **\dir{.};
 (8,-4)*{x};
 (0,-6)*{\gcl};
 (-6,-8);(6,-8); **\dir{.};
 (10,-8)*{g(x)};
\endxy
\end{equation}
where $\gch$ and $\gcl$ are identity on all strands except those affected by $g$ where there is one simple crossing,
\item if $g$ switches $\f_i$ and $\e_i$
\begin{equation}
\label{gamma2}
\xy
 (-6,5);(6,5); **\dir{.};
 (8,5)*{x};
 (0,0)*{f};
 (-6,-5);(6,-5); **\dir{.};
 (8,-5)*{x};
\endxy\mxto{\gamma}
-\xy
 (-6,8);(6,8); **\dir{.};
 (10,8)*{g(x)};
 (0,6)*{\gch};
 (-6,4);(6,4); **\dir{.};
 (8,4)*{x};
 (0,0)*{f};
 (-6,-4);(6,-4); **\dir{.};
 (8,-4)*{x};
 (0,-6)*{\gcl};
 (-6,-8);(6,-8); **\dir{.};
 (10,-8)*{g(x)};
\endxy+\sum_{f_1+f_2+f_3=\nu_i-1}q^{-\nu_i+1+2f_3}
\xy
 (-8,12);(8,12); **\dir{.};
 (12,12)*{h(x)};
 (0,8)*{\gth};
 (-8,4);(8,4); **\dir{.};
 (20,4)*{x\la \nu_i-1-2f_3\ra};
 (0,0)*{f};
 (-8,-4);(8,-4); **\dir{.};
 (20,-4)*{x\la \nu_i-1-2f_3\ra};
 (0,-8)*{\gtl};
 (-8,-12);(8,-12); **\dir{.};
 (12,-12)*{h(x)};
\endxy
\end{equation}
where all additional 2-morphisms are the identity on all strands except for those effected by $g$, and $\nu$ is the label of the area next to them.
\end{itemize}

Note that, using relations \eqref{DubbleCrossijr} and \eqref{DubbleCrossiir1}, the map $\gamma$ is constructed precisely in the way such that $[\gamma(f)]=[f]$ in $\Tr(\du)$.

If fact, we are interested only in the restriction $\gamma\colon \Mlm\to \Hlm$. The composition $\beta^\infty\gamma\colon \Mlm\to \Mlm$ simplifies the source 1-morphism $x\in\hhlm$ in the way that it either has smaller length $l$, or has the same length, but has less $\f$-s after $\e$-s, or is the element of $\kk$. More precisely, $\beta^\infty\gamma$ of the base 2-morphism is a linear combination of base 2-morphisms with simpler source 1-morphism. Therefore, by repeating $\beta^\infty\gamma$ we finally end up in $\Dlm$ and $\beta^\infty\gamma|_\Dlm=\Id_\Dlm$. Similarly as before, we can define a projection $(\beta^\infty\gamma)^\infty\colon \Mlm\to \Dlm$ which is equal to $(\beta^\infty\gamma)^n$ for sufficiently large $n$. It is our third map of \eqref{Maps}.

We can caracterize $\Dlm$ in another way. Chose another object $\nu\in\X$. Let ${_\nu D_\lambda}^+\subset {_\nu M_\lambda}$ be generated by 2-morphisms $f\colon x^+\to x^+$ for $x^+\in{_\nu{\bf h}_\lambda}\cap\kk^+$. Similarly, let ${_\mu D_\nu}^-\subset {_\mu M_\nu}$ be generated by 2-morphisms $f\colon x^-\to x^-$ for $x^-\in{_\mu{\bf h}_\nu}\cap\kk^-$. Then, $\Dlm=\{f^-\circ f^+\in H|\nu\in\X, f^-\in {_\mu D_\nu}^-,f^+\in {_\nu D_\lambda}^+\}$.

${_\nu D_\lambda}^+$ is indeed the endomorphism space of $\up(\lambda,\nu)$. Theorem \ref{dec} ensures that every 1-morphism $x^+$ of $\up$ is isomorphic to the unique linear combination $\sum_ix^+_i$ where $x^+_i\in\b^{*+}$ with coordinates of the isomorphism $x^+\xto{\xi_i}x^+_i\xto{\xi'_i}x^+$. So it is $\sum_i\xi'_i\xi_i=\Id_{x^+}$. Let $\delta^+\colon {_\nu D_\lambda}^+\to {_\nu K_\lambda}^+:=K^+\cap\H(\u(\lambda,\nu))$ act as
\begin{equation}
\xy
 (-6,5);(6,5); **\dir{.};
 (9,5)*{x^+};
 (0,0)*{f^+};
 (-6,-5);(6,-5); **\dir{.};
 (9,-5)*{x^+};
\endxy\quad\mxto{\delta^+}\quad\sum_i
\xy
 (-6,10);(6,10); **\dir{.};
 (9,10)*{x_i^+};
 (0,7)*{\xi_i};
 (-6,4);(6,4); **\dir{.};
 (9,4)*{x^+};
 (0,0)*{f^+};
 (-6,-4);(6,-4); **\dir{.};
 (9,-4)*{x^+};
 (0,-7)*{\xi'_i};
 (-6,-10);(6,-10); **\dir{.};
 (9,-10)*{x_i^+};
\endxy
\end{equation}
Proposition \ref{sup} ensures that elements of the sum are multiples of the identity, hence the result is indeed in $K^+$. Clearly, $[\delta^+(f^+)]=[f^+]$ in $\Tr(\dup)$.

In the same way we define $\delta^-\colon {_\mu D_\nu}^-\to {_\mu K_\nu}^-:=K^-\cap\H(\u(\nu,\mu))$ and let $\delta:=\delta^-\circ\delta^+\colon \Dlm\to \Klm$, $f^-\circ f^+\mapsto \delta^-(f^-)\circ\delta^+(f^+)$. Clearly, $[\delta(f)]=[f]$ in $\Tr(\du)$. It is our last map of \eqref{Maps}.

No we can fill \eqref{Maps} with maps:
\begin{equation*}
\dHlm\, \xto{\alpha}\, \Hlm\, \xto{\beta^\infty}\, \Mlm\, 
\xy
 (6,0);(-6,0); **\dir{-}?(1)*\dir{>};
 (0,2)*{\scs (\beta^\infty\gamma)^\infty};
\endxy
\, \Dlm\, \xto{\delta}\, \Klm
\end{equation*}
and define $\pi\colon\dHlm\to \Klm$:
\begin{equation}
\pi:=\delta(\beta^\infty\gamma)^\infty\beta^\infty\alpha.
\end{equation}
It holds $[\pi(f)]=[f]$ in $\Tr(\du)$ for every $f\in\dHlm$.

\subsection{Trace property of the projection $\pi$}
\label{TraceP}

In what follows
\begin{equation*}
\left\{\xy
 (0,9)*{};
 (-6,8);(6,8); **\dir{--};
 (8,8)*{x};
 (0,4)*{a};
 (-6,0);(6,0); **\dir{--};
 (8,0)*{y};
 (0,-4)*{b};
 (-6,-8);(6,-8); **\dir{--};
 (8,-8)*{x};
 (0,-9)*{};
\endxy\right\}
\end{equation*}
means $\pi\left(\vc{a}{b}{x}{y}{x}\right)=\pi\left(\vc{b}{a}{y}{x}{y}\right)$ for 2-morphisms $a\colon y\to x$ and $b\colon x\to y$ in $\du(\lambda,\mu)$. We can omit 1-morphisms $x$ and $y$, but not the lines. Note that
\begin{equation}
\left\{\xy
 (0,9)*{};
 (-6,8);(6,8); **\dir{--};
 (0,4)*{a};
 (-6,0);(6,0); **\dir{--};
 (0,-4)*{b};
 (-6,-8);(6,-8); **\dir{--};
 (0,-9)*{};
\endxy\right\}\Leftrightarrow
\left\{\xy
 (0,9)*{};
 (-6,8);(6,8); **\dir{--};
 (0,4)*{b};
 (-6,0);(6,0); **\dir{--};
 (0,-4)*{a};
 (-6,-8);(6,-8); **\dir{--};
 (0,-9)*{};
\endxy\right\}, \quad\quad\quad
\left\{\xy
 (0,10)*{};
 (-6,9);(6,9); **\dir{--};
 (0,6)*{a};
 (-6,3);(6,3); **\dir{--};
 (0,0)*{b};
 (0,-6)*{c};
 (-6,-9);(6,-9); **\dir{--};
 (0,-10)*{};
\endxy\right\}\text{ and }
\left\{\xy
 (0,10)*{};
 (-6,9);(6,9); **\dir{--};
 (0,6)*{a};
 (0,0)*{b};
 (-6,-3);(6,-3); **\dir{--};
 (0,-6)*{c};
 (-6,-9);(6,-9); **\dir{--};
 (0,-10)*{};
\endxy\right\}\Rightarrow
\left\{\xy
 (0,10)*{};
 (-6,9);(6,9); **\dir{--};
 (0,6)*{b};
 (-6,3);(6,3); **\dir{--};
 (0,0)*{c};
 (0,-6)*{a};
 (-6,-9);(6,-9); **\dir{--};
 (0,-10)*{};
\endxy\right\}.
\end{equation}

For $x,y\in\hh$ we call a 2-morphism $a\colon x\to y$ drawable if there is a single diagram (not a linear combination) representing $a$.

\begin{prop}
For every 2-morphisms $a\colon x\to y$ and $b\colon y\to x$, $x,y\in\hhlm$, holds
\begin{equation*}
\left\{\xy
 (0,9)*{};
 (-6,8);(6,8); **\dir{--};
 (0,4)*{a};
 (-6,0);(6,0); **\dir{--};
 (0,-4)*{b};
 (-6,-8);(6,-8); **\dir{--};
 (0,-9)*{};
\endxy\right\}.
\end{equation*}
\end{prop}
\begin{proof}
It is enough to prove the Proposition for base 2-morphisms $a$ and $b$. Clearly, all base elements are drawable.

If there are bubbles in $a$ or $b$, $\beta\left(\vcs{a}{b}\right)=\beta\left(\vcs{b}{a}\right)=0$, what implies $\tps{a}{b}$. So we can assume there are no bubbles in $a$ and $b$.

First suppose $\vcs{a}{b},\vcs{b}{a}\in \Dlm$. Then all downward arrows in $a$ and $b$ are before all upward arrows, so we can write $a=a^-a^+$, $b=b^-b^+$ where $a^-\colon y^-\to x^-$, $a^+\colon y^+\to x^+$, $b^-\colon x^-\to y^-$ and $b^+\colon x^+\to y^+$.

It is $x^+\cong\sum_ix^+_i$, $x^+_i\in\b^{*+}$ with coordinates of the isomorphism $x^+\xto{\xi_i}x^+_i\xto{\xi'_i}x^+$. Similarly, $y^+\cong\sum_jy^+_i$, $y^+_j\in\b^{*+}$ with coordinates of the isomorphism $y^+\xto{\nu_j}y^+_j\xto{\nu'_j}y^+$. We have
\begin{equation*}
\pi\left(\vcs{a^+}{b^+}\right)=
\delta^+\left(\vcs{a^+}{b^+}\right)=\sum_{i,j}
\xy
 (-6,16);(6,16); **\dir{.};
 (9,16)*{x_i^+};
 (0,13)*{\xi_i};
 (0,8)*{a^+};
 (0,3)*{\nu'_j};
 (-6,0);(6,0); **\dir{.};
 (9,0)*{y_j^+};
 (0,-3)*{\nu_j};
 (0,-8)*{b^+};
 (0,-13)*{\xi'_i};
 (-6,-16);(6,-16); **\dir{.};
 (9,-16)*{x_i^+};
\endxy, \quad\quad
\pi\left(\vcs{b^+}{a^+}\right)=
\delta^+\left(\vcs{b^+}{a^+}\right)=\sum_{i,j}
\xy
 (-6,16);(6,16); **\dir{.};
 (9,16)*{y_j^+};
 (0,13)*{\nu_j};
 (0,8)*{b^+};
 (0,3)*{\xi'_i};
 (-6,0);(6,0); **\dir{.};
 (9,0)*{x_i^+};
 (0,-3)*{\xi_i};
 (0,-8)*{a^+};
 (0,-13)*{\nu'_j};
 (-6,-16);(6,-16); **\dir{.};
 (9,-16)*{y_j^+};
\endxy
\end{equation*}
Because of the Corollary \ref{TrUpH} (1), elements of the upper sums are non-zero only if $x_i=y_j$, and then
\begin{equation*}
\xy
 (-6,8);(6,8); **\dir{.};
 (9,8)*{x_i^+};
 (0,5)*{\xi_i};
 (0,0)*{a^+};
 (0,-5)*{\nu'_j};
 (-6,-8);(6,-8); **\dir{.};
 (9,-8)*{y_j^+};
\endxy=c_{i,j}\Id_{x_i} \quad\quad\quad
\xy
 (-6,8);(6,8); **\dir{.};
 (9,8)*{y_j^+};
 (0,5)*{\nu_j};
 (0,0)*{b^+};
 (0,-5)*{\xi'_i};
 (-6,-8);(6,-8); **\dir{.};
 (9,-8)*{x_i^+};
\endxy=d_{i,j}\Id_{x_i}
\end{equation*}
for constants $c_{i,j},d_{i,j}\in\Z$. We make $c_{i,j}=d_{i,j}=0$ if $x_i\neq y_j$. Therefore,
\begin{equation*}
\pi\left(\vcs{a^+}{b^+}\right)=\sum_{i,j}c_{i,j}d_{i,j}\Id_{x_i}=\pi\left(\vcs{b^+}{a^+}\right),
\end{equation*}
so $\tp{a^+}{b^+}{x^+}{y^+}$. Similarly, $\tp{a^-}{b^-}{x^-}{y^-}$, and therefore $\tp{a}{b}{x}{y}$.

We continue the proof by induction on length of the longer source 1-morphism $l(x)$: suppose $\tp{a}{b}{x}{y}$ if $l(x),l(y)<n$. We need to prove the statement when $\max(l(x),l(y))=n$.

Let us do it first for $\vcs{a}{b}\notin \Mlm$, i.e.\ there are turns in $a$ and/or $b$. This we do by another induction, on total number of crossings in $a$ and $b$: suppose $\tp{a}{b}{x}{y}$ if there are less than $k$ crossings of strands in $a$ and $b$ together. We need to prove the statement when $\max(l(x),l(y))=n$ and $a$ and $b$ have exactly $k$ crossings. Without loss of generality, let $l(x)=n$. We need three lemmas.

\begin{lem}
Let $l(x),l(y)\leq n$ and $a$ and $b$ have together exactly $k$ crossings. Under current induction hypothesis
\begin{equation*}
\left\{\xy
 (0,9)*{};
 (-6,8);(6,8); **\dir{--};
 (0,4)*{a};
 (-6,0);(6,0); **\dir{--};
 (0,-4)*{b};
 (-6,-8);(6,-8); **\dir{--};
 (0,-9)*{};
\endxy\right\}\Leftrightarrow
\left\{\xy
 (0,9)*{};
 (-6,8);(6,8); **\dir{--};
 (0,4)*{a'};
 (-6,0);(6,0); **\dir{--};
 (0,-4)*{b};
 (-6,-8);(6,-8); **\dir{--};
 (0,-9)*{};
\endxy\right\}
\end{equation*}
if we transform $a$ to $a'$ using local moves
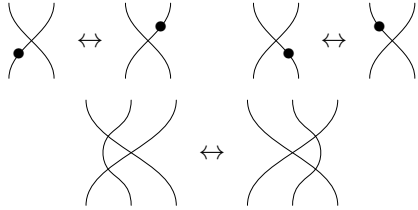
\begin{equation*}
\xy
  (0,0)*{\xybox{
    (-3,-5)*{};(3,5)*{} **\crv{(-3,-1) & (3,1)}?(1)?(.3)*{\bullet};
    (3,-5)*{};(-3,5)*{} **\crv{(3,-1) & (-3,1)}?(1);
    (-5,0)*{};(5,0)*{};
  }};
\endxy\leftrightarrow
\xy
  (0,0)*{\xybox{
    (-3,-5)*{};(3,5)*{} **\crv{(-3,-1) & (3,1)}?(1)?(.7)*{\bullet};
    (3,-5)*{};(-3,5)*{} **\crv{(3,-1) & (-3,1)}?(1);
    (-5,0)*{};(5,0)*{};
  }};
\endxy\quad\quad
\xy
  (0,0)*{\xybox{
    (-3,-5)*{};(3,5)*{} **\crv{(-3,-1) & (3,1)}?(1);
    (3,-5)*{};(-3,5)*{} **\crv{(3,-1) & (-3,1)}?(1)?(.3)*{\bullet};
    (-5,0)*{};(5,0)*{};
  }};
\endxy\leftrightarrow
\xy
  (0,0)*{\xybox{
    (-3,-5)*{};(3,5)*{} **\crv{(-3,-1) & (3,1)}?(1);
    (3,-5)*{};(-3,5)*{} **\crv{(3,-1) & (-3,1)}?(1)?(.7)*{\bullet};
    (-5,0)*{};(5,0)*{};
  }};
\endxy
\end{equation*}
\begin{equation*}
\xy
 (-6,-7)*{};(6,7)*{} **\crv{(-6,-1) & (6,1)};
 (6,-7)*{};(-6,7)*{} **\crv{(6,-1) & (-6,1)};
 (0,-7)*{};(0,7)*{} **\crv{(0,-4) & (-4,-3) & (-4,3) & (0,4)};
 (-8,0)*{};(8,0)*{};
\endxy\leftrightarrow
\xy
 (-6,-7)*{};(6,7)*{} **\crv{(-6,-1) & (6,1)};
 (6,-7)*{};(-6,7)*{} **\crv{(6,-1) & (-6,1)};
 (0,-7)*{};(0,7)*{} **\crv{(0,-4) & (4,-3) & (4,3) & (0,4)};
 (-8,0)*{};(8,0)*{};
\endxy
\end{equation*}
for any colouring and orientation of strands.
\end{lem}
\begin{proof}
Regardless of the colouring and the orientation of the strands it is
\begin{equation*}
\xy
  (0,0)*{\xybox{
    (-3,-5)*{};(3,5)*{} **\crv{(-3,-1) & (3,1)}?(1)?(.3)*{\bullet};
    (3,-5)*{};(-3,5)*{} **\crv{(3,-1) & (-3,1)}?(1);
    (-5,0)*{};(5,0)*{};
  }};
\endxy=
\xy
  (0,0)*{\xybox{
    (-3,-5)*{};(3,5)*{} **\crv{(-3,-1) & (3,1)}?(1)?(.7)*{\bullet};
    (3,-5)*{};(-3,5)*{} **\crv{(3,-1) & (-3,1)}?(1);
    (-5,0)*{};(5,0)*{};
  }};
\endxy + g
\end{equation*}
where $g$ is something without the crossing, if not $0$ (see \eqref{BullCrossii} and \eqref{BullCrossij}). Therefore $a=a'+G$ where $G$ has one crossing less than $a$, so by induction hypothesis $\tps{G}{b}$ holds. It is
\begin{equation*}
\pi\left(\vcs{a}{b}\right)=\pi\left(\vcs{a'}{b}\right)+\pi\left(\vcs{G}{b}\right),
\end{equation*}
\begin{equation*}
\pi\left(\vcs{b}{a}\right)=\pi\left(\vcs{b}{a'}\right)+\pi\left(\vcs{b}{G}\right).
\end{equation*}
By subtracting upper equations we get $\tps{a}{b}\Leftrightarrow\tps{a'}{b}$. Other moves are similar.
\end{proof}

\begin{lem}
Let $l(x),l(y)\leq n$ and $a$ and $b$ have together exactly $k$ crossings. Under current induction hypothesis
\begin{equation*}
\left\{\xy
 (0,9)*{};
 (-6,8);(6,8); **\dir{--};
 (0,4)*{a};
 (-6,0);(6,0); **\dir{--};
 (0,-4)*{b};
 (-6,-8);(6,-8); **\dir{--};
 (0,-9)*{};
\endxy\right\}
\end{equation*}
is true if $a$ or $b$ contains
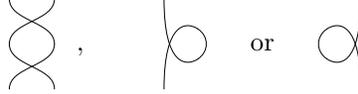
\begin{equation*}
\xy
(-3,0)*{};(3,6)*{} **\crv{(-3,3) & (3,3)};
(3,0)*{};(-3,6)*{} **\crv{(3,3) & (-3,3)};
(-3,-6)*{};(3,0)*{} **\crv{(-3,-3) & (3,-3)};
(3,-6)*{};(-3,0)*{} **\crv{(3,-3) & (-3,-3)};
(-6,0)*{};(6,0)*{};
\endxy,\quad\quad
\xy
(0,-6)*{};(0,6)*{} **\crv{(0,3) & (6,3) & (6,-3) & (0,-3)};
(-3,0)*{};(8,0)*{};
\endxy\quad\text{or}\quad
\xy
(0,-6)*{};(0,6)*{} **\crv{(0,3) & (-6,3) & (-6,-3) & (0,-3)};
(3,0)*{};(-8,0)*{};
\endxy
\end{equation*}
for any colouring and orientation of strands.
\end{lem}
\begin{proof}
Regardless of the colouring and the orientation of the strands
\begin{equation*}
\xy
(-3,0)*{};(3,6)*{} **\crv{(-3,3) & (3,3)};
(3,0)*{};(-3,6)*{} **\crv{(3,3) & (-3,3)};
(-3,-6)*{};(3,0)*{} **\crv{(-3,-3) & (3,-3)};
(3,-6)*{};(-3,0)*{} **\crv{(3,-3) & (-3,-3)};
(-6,0)*{};(6,0)*{};
\endxy = g
\end{equation*}
where $g$ is a sum of drawable 2-morphisms without crossings, or $0$ (see \eqref{DubbleCrossii}, \eqref{DubbleCrossiir1}, \eqref{DubbleCrossiir2}, \eqref{DubbleCrossijr} and \eqref{DubbleCrossij}). Therefore $a$ can be shown to be a linear combination of drawable 2-morphisms with fewer crossings than in $a$, so by induction hypothesis $\tps{a}{b}$ holds. The other cases are similar.
\end{proof}

\begin{lem}
Let $l(x)=n$, $l(y)<n$. Under current induction hypothesis
\begin{equation*}
\left\{\xy
 (0,9)*{};
 (-6,8);(6,8); **\dir{--};
 (8,8)*{x};
 (0,4)*{a};
 (-6,0);(6,0); **\dir{--};
 (8,0)*{y};
 (0,-4)*{b};
 (-6,-8);(6,-8); **\dir{--};
 (8,-8)*{x};
 (0,-9)*{};
\endxy\right\}\Leftrightarrow
\left\{\xy
 (0,9)*{};
 (-6,8);(6,8); **\dir{--};
 (8,8)*{x};
 (0,4)*{a'};
 (-6,0);(6,0); **\dir{--};
 (8,0)*{y'};
 (0,-4)*{b'};
 (-6,-8);(6,-8); **\dir{--};
 (8,-8)*{x};
 (0,-9)*{};
\endxy\right\}
\end{equation*}
if we switch from left to right term using local moves across the middle line
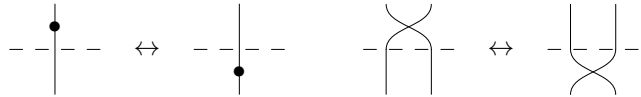
\begin{equation*}
\xy
 (-6,0);(6,0); **\dir{--};
 (0,-6);(0,6); **\dir{-};
 (0,3)*{\bullet};
\endxy\quad\leftrightarrow\quad
\xy
 (-6,0);(6,0); **\dir{--};
 (0,-6);(0,6); **\dir{-};
 (0,-3)*{\bullet};
\endxy\quad\quad\quad
\xy
 (-6,0);(6,0); **\dir{--};
 (-3,0)*{};(3,6)*{} **\crv{(-3,3) & (3,3)};
 (3,0)*{};(-3,6)*{} **\crv{(3,3) & (-3,3)};
 (-3,-6);(-3,0); **\dir{-};
 (3,-6);(3,0); **\dir{-};
\endxy\quad\leftrightarrow\quad
\xy
 (-6,0);(6,0); **\dir{--};
 (-3,0);(-3,6); **\dir{-};
 (3,0);(3,6); **\dir{-};
 (-3,-6)*{};(3,0)*{} **\crv{(-3,-3) & (3,-3)};
 (3,-6)*{};(-3,0)*{} **\crv{(3,-3) & (-3,-3)};
\endxy
\end{equation*}
for any colouring and orientation of strands.
\end{lem}
\begin{proof}
For the first move we need to prove
$\left\{\xy
 (0,6.5)*{};
 (-3,6);(3,6); **\dir{--};
 (5,6)*{x};
 (0,4)*{a};
 (0,0)*{\gb};
 (-3,-2);(3,-2); **\dir{--};
 (5,-2)*{x};
 (0,-4)*{b};
 (-3,-6);(3,-6); **\dir{--};
 (5,-6)*{x};
 (0,-6.5)*{};
\endxy\right\}\Leftrightarrow
\left\{\xy
 (0,6.5)*{};
 (-3,6);(3,6); **\dir{--};
 (5,6)*{x};
 (0,4)*{a};
 (-3,2);(3,2); **\dir{--};
 (5,2)*{y};
 (0,0)*{\gb};
 (0,-4)*{b};
 (-3,-6);(3,-6); **\dir{--};
 (5,-6)*{x};
 (0,-6.5)*{};
\endxy\right\}$.
Indeed,
$\left\{\xy
 (0,6.5)*{};
 (-3,6);(3,6); **\dir{--};
 (5,6)*{y};
 (0,4)*{\gb};
 (-3,2);(3,2); **\dir{--};
 (5,2)*{y};
 (0,0)*{b};
 (0,-4)*{a};
 (-3,-6);(3,-6); **\dir{--};
 (5,-6)*{y};
 (0,-6.5)*{};
\endxy\right\}$
holds by the induction hypothesis because $l(y)<n$, what implies the equivalence. Other move is similar.
\end{proof}

It is enough to prove $\tp{g}{b}{x}{y}$ where $g$ is a generator and $b$ is the base element. We cut $b$ into $\vc{b'}{b''}{y}{z}{x}$ just above turns pointing down, like in \eqref{cut}. Note that $l(z)<n$ because $\vc{g}{b}{x}{y}{x}\notin \Mlm$. There are two cases:
\begin{enumerate}[i)]
\item Let $g=\gb$ or $g=\gc$ for any colouring and orientation of strands. Using moves from three lemmas above
$\left\{\xy
 (0,6.5)*{};
 (-3,6);(3,6); **\dir{--};
 (0,4)*{g};
 (0,0)*{b'};
 (-3,-2);(3,-2); **\dir{--};
 (0,-4)*{b''};
 (-3,-6);(3,-6); **\dir{--};
 (0,-6.5)*{};
\endxy\right\}$
is either true, or equivalent with  $\tps{B'}{B''}$ where $\vcs{B'}{B''}$ is the base element cut as in definition \eqref{cut}, so it is $\beta\left(\vcs{B'}{B''}\right)=\vcs{B''}{B'}$. Acting with $\pi=\delta(\beta^\infty\gamma)^\infty\beta^\infty\alpha$ on both sides, we get $\pi\left(\vcs{B'}{B''}\right)=\pi\left(\vcs{B''}{B'}\right)$, i.e.\ $\tps{B'}{B''}$, so
$\left\{\xy
 (0,6.5)*{};
 (-3,6);(3,6); **\dir{--};
 (0,4)*{g};
 (0,0)*{b'};
 (-3,-2);(3,-2); **\dir{--};
 (0,-4)*{b''};
 (-3,-6);(3,-6); **\dir{--};
 (0,-6.5)*{};
\endxy\right\}$
is true. Similarly,
$\left\{\xy
 (0,6.5)*{};
 (-3,6);(3,6); **\dir{--};
 (0,4)*{b'};
 (-3,2);(3,2); **\dir{--};
 (0,0)*{b''};
 (0,-4)*{g};
 (-3,-6);(3,-6); **\dir{--};
 (0,-6.5)*{};
\endxy\right\}$
is true, so
$\left\{\xy
 (0,6.5)*{};
 (-3,6);(3,6); **\dir{--};
 (0,4)*{g};
 (-3,2);(3,2); **\dir{--};
 (0,0)*{b'};
 (0,-4)*{b''};
 (-3,-6);(3,-6); **\dir{--};
 (0,-6.5)*{};
\endxy\right\}\Leftrightarrow \tps{g}{b}$
is proven.
\item Let $g=\gt$ for any colouring and orientation of strand. 
$\left\{\xy
 (0,6.5)*{};
 (-3,6);(3,6); **\dir{--};
 (0,4)*{\gt};
 (0,0)*{b'};
 (-3,-2);(3,-2); **\dir{--};
 (0,-4)*{b''};
 (-3,-6);(3,-6); **\dir{--};
 (0,-6.5)*{};
\endxy\right\}$
is already the base element $\vcs{B'}{B''}$ cut as in definition \eqref{cut}, so it is true.
$\left\{\xy
 (0,6.5)*{};
 (-3,6);(3,6); **\dir{--};
 (0,4)*{b'};
 (-3,2);(3,2); **\dir{--};
 (0,0)*{b''};
 (0,-4)*{\gt};
 (-3,-6);(3,-6); **\dir{--};
 (0,-6.5)*{};
\endxy\right\}$
is true by induction hypothesis (both source 1-morphisms have length less than $n$), so $\tps{\gt}{b}$ holds. For 
$g=\xy
(-2,-1)*{};(2,-1)*{} **\crv{(-2,2) & (2,2)};
\endxy$
the proof is the same, just turned upside down.
\end{enumerate}

We still need to prove the proposition for $\vcs{a}{b}\in \Mlm$ but $\vcs{a}{b}\notin \Dlm$ or $\vcs{b}{a}\notin \Dlm$, i.e.\ $x\notin\kk$ or $y\notin\kk$. under induction on $\max(l(x),l(y))$. Note that $\vcs{a}{b}\in \Mlm$ implies $l(x)=l(y)$.

\begin{lem}
Let $l(x)=l(y)=n$ and $x\notin\kk$. Under current induction hypothesis
\begin{equation*}
\left\{\xy
 (0,9)*{};
 (-6,8);(6,8); **\dir{--};
 (8,8)*{x};
 (0,4)*{a};
 (-6,0);(6,0); **\dir{--};
 (8,0)*{y};
 (0,-4)*{b};
 (-6,-8);(6,-8); **\dir{--};
 (8,-8)*{x};
 (0,-9)*{};
\endxy\right\}\Leftrightarrow
\left\{\xy
 (0,11)*{};
 (-6,10);(6,10); **\dir{--};
 (10,10)*{g(x)};
 (0,8)*{\gch};
 (0,3)*{a};
 (-6,0);(6,0); **\dir{--};
 (10,0)*{y};
 (0,-3)*{b};
 (0,-8)*{\gcl};
 (-6,-10);(6,-10); **\dir{--};
 (10,-10)*{g(x)};
 (0,-11)*{};
\endxy\right\}
\end{equation*}
where $\gch$ and $\gcl$ are identity on all strands except those effected by $g$ where is a single crossing, like in \eqref{gamma1} or \eqref{gamma2}.
\end{lem}
\begin{proof}
We write definitions \eqref{gamma1} and \eqref{gamma2} in the general expression
\begin{equation*}
\gamma\left(\xy
 (0,9)*{};
 (-6,8);(6,8); **\dir{.};
 (8,8)*{x};
 (0,4)*{a};
 (-6,0);(6,0); **\dir{.};
 (8,0)*{y};
 (0,-4)*{b};
 (-6,-8);(6,-8); **\dir{.};
 (8,-8)*{x};
 (0,-9)*{};
\endxy\right)=
\pm\quad\xy
 (-6,10);(6,10); **\dir{.};
 (10,10)*{g(x)};
 (0,8)*{\gch};
 (-6,6);(6,6); **\dir{.};
 (8,6)*{x};
 (0,3)*{a};
 (-6,0);(6,0); **\dir{.};
 (8,0)*{y};
 (0,-3)*{b};
 (-6,-6);(6,-6); **\dir{.};
 (8,-6)*{x};
 (0,-8)*{\gcl};
 (-6,-10);(6,-10); **\dir{.};
 (10,-10)*{g(x)};
\endxy+\quad k\sum_i\quad
\xy
 (-6,10);(6,10); **\dir{.};
 (10,10)*{h(x)};
 (0,8)*{g'_i};
 (-6,6);(6,6); **\dir{.};
 (8,6)*{x};
 (0,3)*{a};
 (-6,0);(6,0); **\dir{.};
 (8,0)*{y};
 (0,-3)*{b};
 (-6,-6);(6,-6); **\dir{.};
 (8,-6)*{x};
 (0,-8)*{g_i};
 (-6,-10);(6,-10); **\dir{.};
 (10,-10)*{h(x)};
\endxy
\end{equation*}
where $\pm \, \vcs{\gcl}{\gch}\, +k\sum_i \, \vcs{g_i}{g'_i}=\Id_x$, so
\begin{equation*}
\xy
 (0,9)*{};
 (-6,8);(6,8); **\dir{.};
 (8,8)*{y};
 (0,4)*{b};
 (-6,0);(6,0); **\dir{.};
 (8,0)*{x};
 (0,-4)*{a};
 (-6,-8);(6,-8); **\dir{.};
 (8,-8)*{y};
 (0,-9)*{};
\endxy\quad=\pm\quad
\xy
 (-6,10);(6,10); **\dir{.};
 (8,10)*{y};
 (0,7)*{b};
 (-6,4);(6,4); **\dir{.};
 (8,4)*{x};
 (0,2)*{\gcl};
 (-6,0);(6,0); **\dir{.};
 (10,0)*{g(x)};
 (0,-2)*{\gch};
 (-6,-4);(6,-4); **\dir{.};
 (8,-4)*{x};
 (0,-7)*{a};
 (-6,-10);(6,-10); **\dir{.};
 (8,-10)*{y};
\endxy\quad + k\sum_i\quad
\xy
 (-6,10);(6,10); **\dir{.};
 (8,10)*{y};
 (0,7)*{b};
 (-6,4);(6,4); **\dir{.};
 (8,4)*{x};
 (0,2)*{g_i};
 (-6,0);(6,0); **\dir{.};
 (10,0)*{h(x)};
 (0,-2)*{g'_i};
 (-6,-4);(6,-4); **\dir{.};
 (8,-4)*{x};
 (0,-7)*{a};
 (-6,-10);(6,-10); **\dir{.};
 (8,-10)*{y};
\endxy
\end{equation*}
Note that last term is in $\Mlm$, so
$\left\{\xy
 (0,8.5)*{};
 (-3,8);(3,8); **\dir{--};
 (0,6)*{g'_i};
 (0,2)*{a};
 (-3,0);(3,0); **\dir{--};
 (0,-2)*{b};
 (0,-6)*{g_i};
 (-3,-8);(3,-8); **\dir{--};
 (0,-8.5)*{};
\endxy\right\}$.
By acting with $\pi$ on two upper relations and subtracting them, we get the equivalence.
\end{proof}

Using the lemma we can move $x$ and $y$ to $\kk$, i.e.\ get $\vcs{a}{b},\vcs{b}{a}\in \Dlm$ for which we have already proven the proposition, what concludes the proof of the proposition.
\end{proof}

\begin{cor}
\label{TrUH}
For every par of 2-morphisms $(e,a,e')\colon (x,e)\to (y,e')$ and $(e',b,e)\colon (y,e')\to (x,e)$ in $\du$ it holds
\begin{equation*}
\left\{\xy
 (0,9)*{};
 (-8,8);(8,8); **\dir{--};
 (0,4)*{(e,a,e')};
 (-8,0);(8,0); **\dir{--};
 (0,-4)*{(e',b,e)};
 (-8,-8);(8,-8); **\dir{--};
 (0,-9)*{};
\endxy\right\}.
\end{equation*}
\end{cor}
\begin{proof}
It holds $\vcsl{(e,a,e')}{(e',b,e)}=\vcsl{(e,a,\Id)}{(\Id,b,e)}$ and $\alpha\left(\vcsl{(e,a,\Id)}{(\Id,b,e)}\right)=\vcsl{(\Id,a,\Id)}{(\Id,b,\Id)}=\vcs{a}{b}$. Similarly, $\alpha\left(\vcsl{(e',b,e)}{(e,a,e')}\right)=\vcs{b}{a}$. Acting with $\pi$ on both equalities concludes the proof.
\end{proof}

\subsection{Checking the relations and the final proof}

We have already defined $\rE_i\onel$ and $\rEia\onel$ in $\Tr(\dup)\subset\Tr(\du)$ (see \eqref{Ei}). Let
\begin{equation}
\rF_i\onel:=\left[\Id_{\f_i\onel}\right],\quad\quad
\rFia\onel:=\left[\Id_{\fia\onel}\right],
\end{equation}
in $\Tr(\du)$ for $i\in \I$, $a\in \X$, $t\in\Z$.

Note that in $\Tr(\du)$
\begin{equation}
\left[\xy
 (0,8)*{};
 (-6,8);(6,8); **\dir{.};
 (0,4)*{a};
 (-6,0);(6,0); **\dir{.};
 (0,-4)*{b};
 (-6,-8);(6,-8); **\dir{.};
 (0,-10)*{};
\endxy\right]=q^t
\left[\xy
 (0,8)*{};
 (-6,8);(6,8); **\dir{.};
 (0,4)*{b};
 (-6,0);(6,0); **\dir{.};
 (0,-4)*{a};
 (-6,-8);(6,-8); **\dir{.};
 (0,-10)*{};
\endxy\right]
\end{equation}
where $t=\deg(b)=-\deg(a)$ and that $[c\circ d]=0$ if $\deg(c)\neq 0$.

\begin{lem}
\begin{equation}
\label{Lem1}
\left[\xy
(-3,-5);(-3,5); **\dir{-} ?(0)*\dir{<};
(3,-5);(3,5); **\dir{-} ?(0)*\dir{<};
(-5,0)*{};(5,0)*{};
\endxy\right] = 
\left(1+q^2\right)\left[\xy
(0,0)*{\xybox{
 (-3,-5)*{};(3,5)*{} **\crv{(-3,-1) & (3,1)}?(1)*\dir{>};
 (3,-5)*{};(-3,5)*{} **\crv{(3,-1) & (-3,1)}?(1)*\dir{>}?(.7)*{\bullet};
 (-5,0)*{};(5,0)*{};
}};
\endxy\right] =
-\left(1+q^2\right)\left[\xy
(0,0)*{\xybox{
 (-3,-5)*{};(3,5)*{} **\crv{(-3,-1) & (3,1)}?(1)*\dir{>}?(.7)*{\bullet};
 (3,-5)*{};(-3,5)*{} **\crv{(3,-1) & (-3,1)}?(1)*\dir{>};
 (-5,0)*{};(5,0)*{};
}};
\endxy\right]
\end{equation}
\end{lem}
\begin{proof}
It holds
\begin{equation*}
\left[\xy
(0,0)*{\xybox{
 (-3,-5)*{};(3,5)*{} **\crv{(-3,-1) & (3,1)}?(1)*\dir{>};
 (3,-5)*{};(-3,5)*{} **\crv{(3,-1) & (-3,1)}?(1)*\dir{>}?(.7)*{\bullet};
 (-5,0)*{};(5,0)*{};
}};
\endxy\right] \eqr{\eqref{DubbleCrossii},\eqref{BullCrossii}}
-\left[\xy
 (-3,0)*{};(3,6)*{} **\crv{(-3,3) & (3,3)}?(1)*\dir{>};
 (3,0)*{};(-3,6)*{} **\crv{(3,3) & (-3,3)}?(1)*\dir{>}?(.7)*{\bullet};
 (3,0)*{\bullet};
 (-4.5,-1.5);(5,-1.5); **\dir{.};
 (-3,-6)*{};(3,0)*{} **\crv{(-3,-3) & (3,-3)};
 (3,-6)*{};(-3,0)*{} **\crv{(3,-3) & (-3,-3)};
 (-5,0)*{};(5,0)*{};
\endxy\right] =
-q^{-2}\left[\xy
 (-3,0)*{};(3,6)*{} **\crv{(-3,3) & (3,3)}?(1)*\dir{>};
 (3,0)*{};(-3,6)*{} **\crv{(3,3) & (-3,3)}?(1)*\dir{>};
 (-3,0)*{\bullet};
 (-3,-6)*{};(3,0)*{} **\crv{(-3,-3) & (3,-3)};
 (3,-6)*{};(-3,0)*{} **\crv{(3,-3) & (-3,-3)}?(.3)*{\bullet};
 (-5,0)*{};(5,0)*{};
\endxy\right] \eqr{\eqref{DubbleCrossii},\eqref{BullCrossii}}
-q^{-2}\left[\xy
(0,0)*{\xybox{
 (-3,-5)*{};(3,5)*{} **\crv{(-3,-1) & (3,1)}?(1)*\dir{>};
 (3,-5)*{};(-3,5)*{} **\crv{(3,-1) & (-3,1)}?(1)*\dir{>}?(.3)*{\bullet};
 (-5,0)*{};(5,0)*{};
}};
\endxy\right]
\end{equation*}
so
\begin{equation*}
\left(1+q^2\right)\left[\xy
(0,0)*{\xybox{
 (-3,-5)*{};(3,5)*{} **\crv{(-3,-1) & (3,1)}?(1)*\dir{>};
 (3,-5)*{};(-3,5)*{} **\crv{(3,-1) & (-3,1)}?(1)*\dir{>}?(.7)*{\bullet};
 (-5,0)*{};(5,0)*{};
}};
\endxy\right] =
\left[\xy
(0,0)*{\xybox{
 (-3,-5)*{};(3,5)*{} **\crv{(-3,-1) & (3,1)}?(1)*\dir{>};
 (3,-5)*{};(-3,5)*{} **\crv{(3,-1) & (-3,1)}?(1)*\dir{>}?(.7)*{\bullet};
 (-5,0)*{};(5,0)*{};
}};
\endxy\right] -
\left[\xy
(0,0)*{\xybox{
 (-3,-5)*{};(3,5)*{} **\crv{(-3,-1) & (3,1)}?(1)*\dir{>};
 (3,-5)*{};(-3,5)*{} **\crv{(3,-1) & (-3,1)}?(1)*\dir{>}?(.3)*{\bullet};
 (-5,0)*{};(5,0)*{};
}};
\endxy\right] \eqr{\eqref{BullCrossii}}
\left[\xy
(-3,-5);(-3,5); **\dir{-} ?(0)*\dir{<};
(3,-5);(3,5); **\dir{-} ?(0)*\dir{<};
(-5,0)*{};(5,0)*{};
\endxy\right].
\end{equation*}
Other relation is shown in a similar way.
\end{proof}

\begin{lem}
\label{rel1}
In $\Tr(\du)$
\begin{equation}
\rE_i^2\rE_j\onel-\left(q+\q\right)\rE_i\rE_j\rE_i\onel+\rE_j\rE_i^2\onel = 0 \quad \text{ if } j\neq i,
\end{equation}
\begin{equation}
\rF_i^2\rF_j\onel-\left(q+\q\right)\rF_i\rF_j\rF_i\onel+\rF_j\rF_i^2\onel = 0 \quad \text{ if } j\neq i.
\end{equation}
\end{lem}
\begin{proof}
\begin{equation*}
 \rE_2\rE_1^2\onel + \rE_1^2\rE_2\onel \eqr{\eqref{Lem1}}
\left(1+q^2\right)\left[\xy
(0,0)*{\xybox{
 (-9,-5);(-9,5); **[blue]\dir{-} ?(0)*[blue]\dir{<};
 (-3,-5)*{};(3,5)*{} **[red]\crv{(-3,-1) & (3,1)}?(1)*[red]\dir{>};
 (3,-5)*{};(-3,5)*{} **[red]\crv{(3,-1) & (-3,1)}?(1)*[red]\dir{>}?(.7)*{\bullet};
 (-11,0)*{};(5,0)*{};
}};
\endxy\right]
- \left(1+q^2\right)
\left[\xy
(0,0)*{\xybox{
 (-3,-5)*{};(3,5)*{} **[red]\crv{(-3,-1) & (3,1)}?(1)*[red]\dir{>}?(.7)*{\bullet};
 (3,-5)*{};(-3,5)*{} **[red]\crv{(3,-1) & (-3,1)}?(1)*[red]\dir{>};
 (9,-5);(9,5); **[blue]\dir{-} ?(0)*[blue]\dir{<};
 (-5,0)*{};(11,0)*{};
}};
\endxy\right] = 
\end{equation*}
\begin{equation*}
\eqr{\eqref{DubbleCrossij}} \left(1+q^2\right)\left[\xy
(0,0)*{\xybox{
 (3,-7)*{};(-3,7)*{} **[red]\crv{(3,-3) & (-8,-3) & (-8,4) & (-3,4)}?(1)*[red]\dir{>};
 (-3,-7)*{};(3,7)*{} **[red]\crv{(-3,-1) & (3,1)}?(1)*[red]\dir{>}?;
 (-9,-7)*{};(-9,7)*{} **[blue]\crv{(-9,-3) & (-4,-2) & (-4,4) & (-9,4)}?(1)*[blue]\dir{>};
 (-10.5,1);(5,1); **\dir{.};
 (5,0)*{};(-11,0)*{};
}};
\endxy\right]
-\left(1+q^2\right)\left[\xy
(0,0)*{\xybox{
 (-3,-7)*{};(3,7)*{} **[red]\crv{(-3,-3) & (8,-3) & (8,4) & (3,4)}?(1)*[red]\dir{>};
 (3,-7)*{};(-3,7)*{} **[red]\crv{(3,-1) & (-3,1)}?(1)*[red]\dir{>}?;
 (9,-7)*{};(9,7)*{} **[blue]\crv{(9,-3) & (4,-2) & (4,4) & (9,4)}?(1)*[blue]\dir{>};
 (-4.5,1);(11,1); **\dir{.};
 (-5,0)*{};(11,0)*{};
}};
\endxy\right]=
\end{equation*}
\begin{equation*}
=\left(q+\q\right)\left[\xy
(0,0)*{\xybox{
 (-6,-7)*{};(6,7)*{} **[red]\crv{(-6,-1) & (6,1)}?(1)*[red]\dir{>};
 (6,-7)*{};(-6,7)*{} **[red]\crv{(6,-1) & (-6,1)}?(1)*[red]\dir{>};
 (0,-7)*{};(0,7)*{} **[blue]\crv{(0,-4) & (-4,-3) & (-4,3) & (0,4)}?(1)*[blue]\dir{>};
 (-7,0)*{};(7,0)*{};
}};
\endxy\right]
-\left(q+\q\right)\left[\xy
(0,0)*{\xybox{
 (-6,-7)*{};(6,7)*{} **[red]\crv{(-6,-1) & (6,1)}?(1)*[red]\dir{>};
 (6,-7)*{};(-6,7)*{} **[red]\crv{(6,-1) & (-6,1)}?(1)*[red]\dir{>};
 (0,-7)*{};(0,7)*{} **[blue]\crv{(0,-4) & (4,-3) & (4,3) & (0,4)}?(1)*[blue]\dir{>};
 (-7,0)*{};(7,0)*{};
}};
\endxy\right] \eqr{\eqref{R3iji}}
\left(q+\q\right)\rE_1\rE_2\rE_1\onel.
\end{equation*}
Other relations are shown in a similar way.
\end{proof}

\begin{lem}
\label{rel2}
In $\Tr(\du)$
\begin{equation}
\rE_i\rF_j\onel-\rF_j\rE_i\onel = \delta_{i,j}[\lambda_i]\onel.
\end{equation}
\end{lem}
\begin{proof}
For $i=1$, $j=2$ the relation holds because
\begin{equation*}
\rE_1\rF_2\onel=
\left[\xy
(-3,-6);(-3,6); **[red]\dir{-} ?(0)*[red]\dir{<};
(3,-6);(3,6); **[blue]\dir{-} ?(1)*[blue]\dir{>};
(-5,0)*{};(5,0)*{};
\endxy\right] \eqr{\eqref{DubbleCrossijr}}
\left[\xy
(-3,0)*{};(3,6)*{} **[blue]\crv{(-3,3) & (3,3)};
(3,0)*{};(-3,6)*{} **[red]\crv{(3,3) & (-3,3)}?(1)*[red]\dir{>};
(-3,-6)*{};(3,0)*{} **[red]\crv{(-3,-3) & (3,-3)};
(3,-6)*{};(-3,0)*{} **[blue]\crv{(3,-3) & (-3,-3)}?(0)*[blue]\dir{<};
(-4.5,0);(5,0); **\dir{.};
(-5,0)*{};(5,0)*{};
\endxy\right]=
\left[\xy
(-3,0)*{};(3,6)*{} **[red]\crv{(-3,3) & (3,3)}?(1)*[red]\dir{>};
(3,0)*{};(-3,6)*{} **[blue]\crv{(3,3) & (-3,3)};
(-3,-6)*{};(3,0)*{} **[blue]\crv{(-3,-3) & (3,-3)}?(0)*[blue]\dir{<};
(3,-6)*{};(-3,0)*{} **[red]\crv{(3,-3) & (-3,-3)};
(-5,0)*{};(5,0)*{};
\endxy\right] \eqr{\eqref{DubbleCrossijr}}
\left[\xy
(-3,-6);(-3,6); **[blue]\dir{-} ?(1)*[blue]\dir{>};
(3,-6);(3,6); **[red]\dir{-} ?(0)*[red]\dir{<};
(-5,0)*{};(5,0)*{};
\endxy\right]=
\rF_2\rE_1\onel.
\end{equation*}
For $i=2$, $j=1$ the relation is shown in the similar way. If $i=j$ and $\lambda_i\geq 0$
\begin{equation*}
\rE_i\rF_i\onel=
\left[\xy
{\ar (0,-8)*{}; (0,8)*{}};
(-3,0)*{};(3,0)*{};
\endxy\xy
{\ar (0,8)*{}; (0,-8)*{}};
(-3,-9)*{};(3,9)*{};
\endxy\right] \eqr{\eqref{DubbleCrossiir1}}
-\left[\xy
(-3,0)*{};(3,8)*{} **\crv{(-3,4) & (3,4)};
(3,0)*{};(-3,8)*{} **\crv{(3,4) & (-3,4)}?(1)*\dir{>};
(-3,-8)*{};(3,0)*{} **\crv{(-3,-4) & (3,-4)};
(3,-8)*{};(-3,0)*{} **\crv{(3,-4) & (-3,-4)}?(0)*\dir{<};
(-6,-9)*{};(6,9)*{};
\endxy\right]
+ \sum_{f_1+f_2+f_3=\lambda_i-1}
\left[\xy
 (3,9)*{};(-3,9)*{} **\crv{(3,4) & (-3,4)} ?(1)*\dir{>} ?(.2)*{\bullet};
(5,6)*{\scriptstyle f_1};
 (0,0)*{\rbub};
(5,-2)*{\scriptstyle f_2};
 (-3,-9)*{};(3,-9)*{} **\crv{(-3,-4) & (3,-4)} ?(1)*\dir{>}?(.8)*{\bullet};
(5,-6)*{\scriptstyle f_3};
 (-6,-10)*{};(7,10)*{};
\endxy\right]\, .
\end{equation*}
If the bubble has a non-zero degree, the trace is $0$, so
\begin{equation*}
\rE_i\rF_i\onel=
-\left[\xy
(-3,0)*{};(3,6)*{} **\crv{(-3,3) & (3,3)};
(3,0)*{};(-3,6)*{} **\crv{(3,3) & (-3,3)}?(1)*\dir{>};
(-3,-6)*{};(3,0)*{} **\crv{(-3,-3) & (3,-3)};
(3,-6)*{};(-3,0)*{} **\crv{(3,-3) & (-3,-3)}?(0)*\dir{<};
(-5,0)*{};(5,0)*{};
\endxy\right] + \sum_{f_1+f_3=\lambda_i-1}
\left[\xy
 (3,6)*{};(-3,6)*{} **\crv{(3,1) & (-3,1)} ?(1)*\dir{>} ?(.2)*{\bullet};
(5,3)*{\scriptstyle f_1};
 (-3,-6)*{};(3,-6)*{} **\crv{(-3,-1) & (3,-1)} ?(1)*\dir{>}?(.8)*{\bullet};
(5,-3)*{\scriptstyle f_3};
 (-5.5,0);(7,0); **\dir{.};
 (-6,0)*{};(7,0)*{};
\endxy\right]=
\end{equation*}
\begin{equation*}
=-\left[\xy
(-3,0)*{};(3,6)*{} **\crv{(-3,3) & (3,3)};
(3,0)*{};(-3,6)*{} **\crv{(3,3) & (-3,3)}?(1)*\dir{>};
(-3,-6)*{};(3,0)*{} **\crv{(-3,-3) & (3,-3)};
(3,-6)*{};(-3,0)*{} **\crv{(3,-3) & (-3,-3)}?(0)*\dir{<};
(-5,0)*{};(5,0)*{};
\endxy\right] + \sum_{f_1+f_3=\lambda_i-1}q^{1-\lambda_i+2f_3}
\left[\xy
(0,0)*{\lbbub};
(7,-3)*{\scriptstyle f_1+f_3};
 (-5,-6)*{};(7,6)*{};
\endxy\right] \eqr{\eqref{Zerobub}}
-\left[\xy
(-3,0)*{};(3,6)*{} **\crv{(-3,3) & (3,3)};
(3,0)*{};(-3,6)*{} **\crv{(3,3) & (-3,3)}?(1)*\dir{>};
(-3,-6)*{};(3,0)*{} **\crv{(-3,-3) & (3,-3)};
(3,-6)*{};(-3,0)*{} **\crv{(3,-3) & (-3,-3)}?(0)*\dir{<};
(-5,0)*{};(5,0)*{};
\endxy\right] + [\lambda_i]\onel.
\end{equation*}
For $\lambda_i\geq 0$
\begin{equation*}
\rF_i\rE_i\onel=
\left[\xy
(-3,-6);(-3,6); **\dir{-} ?(1)*\dir{>};
(3,-6);(3,6); **\dir{-} ?(0)*\dir{<};
(-5,0)*{};(5,0)*{};
\endxy\right] \eqr{\eqref{DubbleCrossiir2}}
-\left[\xy
(-3,0)*{};(3,6)*{} **\crv{(-3,3) & (3,3)}?(1)*\dir{>};
(3,0)*{};(-3,6)*{} **\crv{(3,3) & (-3,3)};
(-3,-6)*{};(3,0)*{} **\crv{(-3,-3) & (3,-3)}?(0)*\dir{<};
(3,-6)*{};(-3,0)*{} **\crv{(3,-3) & (-3,-3)};
(-5,0)*{};(5,0)*{};
\endxy\right]\, ,
\end{equation*}
so $\rE_i\rF_i\onel-\rF_i\rE_i\onel = [\lambda_i]\onel$. If $\lambda_i<0$ the proof is similar.
\end{proof}

Now we are ready to show the main result.

\begin{thm}
\label{Final}
There is an isomorphism
\begin{equation}
\Tr(\u)\cong\AdU.
\end{equation}
of $\Z[q,\q]$-algebras.
\end{thm}
\begin{proof}
Proposition \ref{Kar} implies $\Tr(\u)\cong\Tr(\du)$, so it is enough to prove that $\Tr(\du)\cong\AdU$.

Corollary \ref{TrUH} ensures that for every $\lambda,\mu\in \X$ the projection $\pi\colon\dHlm=\H(\du(\lambda,\mu)|_{\hhlm})\to \Klm$ satisfies the trace property (property 3. in Proposition \ref{mpp}) for the category $\du(\lambda,\mu)|_{\hhlm}$. Also the other conditions needed in Proposition \ref{mpp} are satisfied, so $\Tr(\du(\lambda,\mu)|_{\hhlm})\cong K(\lambda,\mu)\cong 1_\mu(\AdU)1_\lambda$ as a $\Z$-module. Proposition \ref{Dir} implies $\Tr(\du(\lambda,\mu))\cong \Tr(\du(\lambda,\mu)|_{\hhlm})\cong 1_\mu(\AdU)1_\lambda$. Clearly, the isomorphism preserves multiplication with $q$, so $\Tr(\du(\lambda,\mu))\cong 1_\mu(\AdU)1_\lambda$ as a $\Z[q,\q]$-module. Therefore,
\begin{equation*}
\Tr(\du)\cong\bigoplus_{\lambda,\mu\in \X}\Tr(\du(\lambda,\mu))\cong\bigoplus_{\lambda,\mu\in \X}1_\mu (\AdU) 1_\lambda \eqr{\eqref{DSDL}} \AdU
\end{equation*}
as a $\Z[q,\q]$-module.

Let the isomorphism be $[\iota]\colon\AdU\to\Tr(\du)$. It maps $\Eia 1_\lambda\mapsto\rEia\onel$ and $\Fia 1_\lambda\mapsto\rFia\onel$. To prove the theorem we need to check that elements multiply (or compose if we look to it as a 1-category) in the same way.

$\AdU\otimes\Q(q)\cong\dU$ is generated by $E_i1_\lambda$ and $F_i1_\lambda$ for $i\in \I$, $\lambda\in \X$, as $\Q(q)$-algebra. Relations \eqref{deE} and \eqref{deF} imply
\begin{equation}
\label{divT2}
\rE_i^a\onel = [a]!\rEia\onel,\quad
\rF_i^a\onel = [a]!\rFia\onel
\end{equation}
in $\Tr(\du)$. Therefore, $\Tr(\du)\otimes\Q(q)$ is generated by $\rE_i\onel=[\iota](E_i1_\lambda)$ and $\rF_i\onel=[\iota](F_i1_\lambda)$ for $i\in \I$, $\lambda\in \X$, as $\Q(q)$-algebra. Recall the defining relation of $\dU$  \eqref{relI}--\eqref{relFF}. Relations \eqref{relI} and \eqref{relII} hold by construction also in $\Tr(\dup)\otimes\Q(q)$ after applying $[\iota]$. Lemmas \ref{rel1} and \ref{rel2} ensure that after applying $[\iota]$ the other relations hold in $\Tr(\dup)\otimes\Q(q)$ too. Divided powers, from which we construct the bases, are defined from above generators in the same way (compare \eqref{divT2} and \eqref{div}) so it is really
\begin{equation*}
\Tr(\du)\otimes\Q(q)\cong\AdU\otimes\Q(q).
\end{equation*}
Since the multiplications are the same, they are also the same in the $\Z[q,\q]$-subalgebras $\Tr(\dup)$ and $\AdUp$, hence the proof.
\end{proof}

\end{document}